\documentclass[10pt,a4paper]{article}

\usepackage{a4wide}
\usepackage{amsthm}
\usepackage{amsfonts}
\usepackage{amsmath}
\usepackage[english]{babel}
\usepackage[all]{xy}
\usepackage{amssymb}
\usepackage{graphicx}

\usepackage{makeidx}
\makeindex

\newcommand{\N}{\mathbb{N}}
\newcommand{\R}{\mathbb{R}}
\newcommand{\Q}{\mathbb{Q}}
\newcommand{\Z}{\mathbb{Z}}
\newcommand{\C}{\mathbb{C}}
\newcommand{\F}{\mathbb{F}}

\newcommand{\lie}{\mathfrak{g}}
\newcommand{\lien}{\mathfrak{n}}
\newcommand{\lieh}{\mathfrak{h}}

\newcommand{\semi}{\rtimes}
\newcommand{\bs}{\backslash}
\newcommand{\ad}{{\rm ad}}

\DeclareMathOperator{\GL}{GL}
\DeclareMathOperator{\Fix}{Fix}
\DeclareMathOperator{\Aut}{Aut}
\DeclareMathOperator{\Out}{Out}
\DeclareMathOperator{\Inn}{Inn}
\DeclareMathOperator{\Aff}{Aff}
\DeclareMathOperator{\aff}{aff}
\DeclareMathOperator{\End}{End}
\DeclareMathOperator{\Ext}{Ext}

\DeclareMathOperator{\Isom}{Isom}
\DeclareMathOperator{\Ker}{Ker}
\newdir{ >}{{}*!/-7pt/@{>}}

\author{Karel Dekimpe\thanks{Supported by long term structural funding -- Methusalem grant of the Flemish Government. }\\
KU Leuven Campus Kulak Kortrijk, E. Sabbelaan 53, 8500 Kortrijk, Belgium}
\title{\bf A Users' Guide to  Infra-nilmanifolds and Almost--Bieberbach groups}
\date{\today}
\newtheorem{Def}{Definition}[section]
\newtheorem{Ex}{Example}[section]
\newtheorem{Cor}[Def]{Corollary}
\newtheorem{Thm}[Def]{Theorem}
\newtheorem{Prop}[Def]{Proposition}
\newtheorem{Lem}[Def]{Lemma}
\newtheorem*{Rmk}{Remark}
\newtheorem*{Prop*}{Proposition}
\newtheorem*{Lem*}{Lemma}

\newtheorem*{con}{Conjecture}

\begin{document}
\maketitle
\begin{abstract}
The aim of this text is to provide a clear description of the theory of Infra-nilmanifolds and their fundamental groups, the 
almost-Bieberbach groups. For most of the proofs of the results, we refer to the literature. Nevertheless, at certain places 
we offer a somewhat different and new approach and in these cases we provide full proofs. We have tried to keep the prerequisites as minimal as possible. We also provide completely worked out examples, with explicit computations.
\end{abstract}

\section{Flat manifolds and crystallographic groups}
The class of infra-nilmanifolds is a natural generalization of the class of flat manifolds. So before entering the world of  infra-nilmanifolds, let us first explore that of the flat manifolds. In this section we will describe things in quite some detail, since a good understanding of this part is indispensable for the rest of the text. The material in this section is standard and can be found  in e.g.\ \cite{char86-1, szcz12-1, wolf77-1}. We also refer to these references for all the proofs of the results mentioned in this section.

\medskip

Let $\R^n$ be the $n$-dimensional Euclidean space. By $\Isom(\R^n)$, we will denote the group of isometries of Euclidean $n$-space . Any element $f\in \Isom(\R^n)$ can be written as a composition of an orthogonal map $A:\R^n\rightarrow \R^n$ 
and a translation $t_a: \R^n\rightarrow \R^n: r \mapsto a+ r$, for some $a\in \R^n$. The orthogonal map $A$ is a linear map 
and so we can identify $A$ with its matrix representation with respect to the standard orthogonal basis of $\R^n$. After this identification, we can say that $A\in O(n)$, where $O(n)$ is the group of orthogonal $n\times n$ matrices.
It follows that any element $f$ of $\Isom(\R^n)$ can be identified with a tuple $(a,A)\in \R^n \times O(n)$. Now,
let $f=(a,A)$ and $g=(b,B)$ be two isometries, then 
\[ (f\circ g)(r)= f(b+Br)= a+Ab +ABr,\]
from which it follows that $f\circ g$ is represented by the tuple $(a+Ab, AB)$. As a consequence, it follows that 
$\Isom(\R^n)$ is the semidirect product group $\R^n\semi O(n)$. 

\medskip

Recall that to construct the semidirect product\index{semidirect product} of two groups $G$ and $H$, one first chooses a group homomorphism 
$\varphi: H \to \Aut(G)$, then the semidirect product  $G\semi_\varphi H$, or just $G \semi H$ when $\varphi$ is 
clear from the context, is the group where the underlying set is the Cartesian product $G\times H$ and where the 
product is given by 
\[ \forall g_1,g_2 \in G,\; \forall h_1,h_2 \in H:\; (g_1,h_1) (g_2,h_2)= (g_1\varphi(h_1)(g_2) , h_1h_2).\]
It is an easy exercise to show that this construction indeed yields a group. When $\varphi:H\to \Aut(G): h\mapsto 1_G$ is the 
trivial homomorphism, the result of forming the  semidirect product $G\semi H$ is the ordinary direct product $G\times H$ of two groups. Of course, when $G$ and $H$ are Lie groups (or topological groups), one requires the usual continuity in this construction. Note that both $G$ and $H$ can be seen as subgroups of $G\semi H$ by identifying them with 
$G\times \{1\}$ and $\{1\} \times H$ respectively. Under this identification, we have that $G$ is a normal subgroup of $G\semi H$, while in general $H$ is not normal. The projection $p:G\semi H\to H$ is a group homomorphism, with kernel $G$.

\medskip

In the situation above, any orthogonal map $A$ can be seen as a (continuous) homomorphism of the abelian group $\R^n$ and so 
there is an inclusion $O(n)\subseteq \Aut(\R^n)$ and it is this inclusion which is used as the homomorphism $\varphi$ to 
construct the semidirect product $\R^n\semi O(n)$. Analogously, we can construct the affine group $\Aff(\R^n) = \R^n \semi \GL(\R^n)$. So elements of $\Aff(\R^n)$ are tuples $(a,A)\in \R^n\times \GL(\R^n)$ and the product is given by 
\[ \forall a,b\in \R^n, \forall A,B \in \GL(\R^n):\;(a,A) (b,B) = (a + A b , AB).\]
It is obvious that $\Isom(\R^n) \subseteq \Aff(\R^n)$. Of course, also elements of the affine group $\Aff(\R^n)$ can be seen as 
maps of $\R^n$, where just as for the isometry group we have that $(a,A)(r)= a+Ar$ for all $a,r\in\R^n$ and $A \in \GL(\R^n)$.
It follows that there is a left action of $\Aff(\R^n)$ on $\R^n$ given by 
\[ \Aff(\R^n) \times \R^n \to \R^n: \left( (a,A),r\right) \mapsto (a,A) \cdot r = a + A r.\]
As a consequence, also any subgroup $E$ of $\Aff(\R^n)$ (or of $\Isom(\R^n)$) acts on $\R^n$. 

\begin{Rmk}
The groups $\Aff(\R^n)$ and $\Isom(\R^n)$ are Lie groups (using the usual smooth manifold structures on  the underlying sets
$\R^n \times O(n)$ and $\R^n \times \GL(\R^n)$).
\end{Rmk}

As mentioned above when introducing the semidirect product of two groups, we can view $\R^n$ as a normal subgroup of $\Isom(\R^n)$ and $\Aff(\R^n)$. In both cases, this is exactly the group of translations. When $G\subseteq \Aff(\R^n)$, we 
will then write $G\cap \R^n$ for the group of translations in $G$, which is then a normal subgroup of $G$. We will
use $r:\Aff(\R^n)=\R^n\semi \GL(\R^n) \to \GL(\R^n)$ to denote the projection on the second factor. For any 
$G\subseteq \Aff(\R^n)$, we will call $r(G)$ the linear part of $G$ and $G\cap \R^n$ the translational part.
We have that $r(G)\cong G/(G\cap \R^n)$.

\medskip

A flat manifold will be constructed as an orbit space $E\backslash \R^n$, where $E\subseteq \Isom(\R^n)$ is acting properly discontinuously, freely and cocompactly on $\R^n$. For completeness, let us recall the notions we've just mentioned.

\begin{Def} Let $G$ be a group acting (continuously) on a locally compact topological space $X$. 
\begin{enumerate}
\item We say that $G$ acts properly discontinuously\index{properly discontinuous}\index{action!properly discontinuous} on $X$ if for every compact subset $K\subseteq X$ it holds that 
\[ \# \{ g\in G\;|\; g \cdot K \cap K \neq \emptyset \}< \infty.\] 
\item The action of $G$ on $X$ is cocompact\index{cocompact}\index{action!cocompact} if the orbit space $G\backslash X$ is compact.
\item The action is free\index{free}\index{action!free} if $\forall x \in X: \; \{g\in G\;|\; g\cdot x =x \}= \{1\}$, i.e.\ the stabilizer of any $x\in X$ is trivial.
\end{enumerate}
\end{Def}

We are now ready to introduce the class of crystallographic groups
\begin{Def}
A $n$--dimensional crystallographic group\index{crystallographic group}, is a cocompact and discrete subgroup of $\Isom(\R^n)$. 
A Bieberbach group\index{Bieberbach group} is a torsion free crystallographic group.
\end{Def}

\begin{Rmk}
A subgroup $\Gamma$ of $\Isom(\R^n)$ is cocompact, if and only if the action of $\Gamma$ on $\Isom(\R^n)$ by left translations is cocompact, i.e.\ the space of cosets $\Gamma\backslash \Isom(\R^n)$ is compact.
\end{Rmk}

The following proposition gives rise to an alternative description of a crystallographic (Bieberbach) group in terms 
of the induced action on $\R^n$.

\begin{Prop}
Let $\Gamma$ be a subgroup of $\Isom(\R^n)$. Then we have
\begin{enumerate}
\item $\Gamma$ is a discrete subgroup of $\Isom(\R^n)$ $\Leftrightarrow$ The action of $\Gamma$ on $\R^n$ is properly discontinuous.
\item $\Gamma$ is cocompact in $\Isom(\R^n)$ $\Leftrightarrow $ The action of $\Gamma$ on $\R^n$ is cocompact.
\item If $\Gamma$ is a discrete subgroup of $\Isom(\R^n)$, then 
$\Gamma$ is torsion free $\Leftrightarrow$  The action of $\Gamma$ on $\R^n$ is free.
\end{enumerate}
\end{Prop}
From this proposition it follows that a $n$-dimensional crystallographic group $\Gamma$ is a subgroup of $\Isom(\R^n)$ which acts properly discontinuously and cocompactly on $\R^n$. A $n$-dimensional Bieberbach group is a $n$-dimensional   crystallographic group for which the action on $\R^n$ is also free.

\medskip

In case $\Gamma$ is a $n$-dimensional Bieberbach group, the quotient space $M=\Gamma\backslash \R^n$ is a manifold and as $\Gamma\subseteq \Isom (\R^n)$, the manifold locally inherits the metric structure of Euclidean space $\R^n$. It follows that $M$ is a compact Riemannian manifold with constant curvature 0, i.e.\ a compact flat manifold\index{flat manifold}. Conversely, any compact flat $n$-dimensional  manifold can be obtained as a quotient space $\Gamma\backslash \R^n$, where $\Gamma$ is a $n$-dimensional Bieberbach group. Moreover, for Bieberbach groups, the natural projection $p:\R^n \to \Gamma\backslash 
\R^n$ is a covering  with $\Gamma$ as its group of deck transformations and so the fundamental group of the 
flat manifold $\Gamma\backslash \R^n$ is exactly $\Gamma$.

\medskip

Before, we proceed let us present some examples of crystallographic and Bieberbach groups.
\begin{enumerate}
\item Let $\Gamma_1=\{ (z,I_n)\in \Aff(\R^n)\;|\; z\in \Z^n\}$ ($I_n$ will be used to denote the $n\times n$-identity matrix), then $\Gamma_1\cong \Z^n$ consists of pure translations, is 
a discrete and cocompact subgroup of $\Isom(\R^n)$ and hence is a Bieberbach group. The corresponding manifold
$\Gamma_1\backslash \R^n$ is the $n$-dimensional torus.
\item Let $\Gamma_2=\{( z, \pm I_n)\in \Aff(\R^n)\;|\;z \in \Z^n\}$. Then $\Gamma_2$ contains $\Gamma_1$ as a subgroup of
index 2, from which it easily follows that also $\Gamma_2$ is a crystallographic group. However, $\Gamma_2$ is not 
torsion free, so $\Gamma_2$ is not a Bieberbach group.
\item \label{Klein}Let $\Gamma_3$ be the subgroup of $\Isom(\R^2)$, which is 
generated by 
\[ a= (e_1, I_2),\; b=(e_2, I_2) \mbox{ and }\alpha= \left(\left(\begin{array}{c} 0 \\ \frac12 \end{array}\right)\;,
\left( \begin{array}{cc} -1 & 0 \\ 0 & 1 \end{array}\right) \right),\]
where $e_1=(1,0)^T$ and $e_2=(0,1)^T$ (we are  writing elements of $\R^n$ as column vectors).
Note that $\alpha^2=b$. The group generated by $a$ and $b$ is isomorphic to $\Z^2$ and is a subgroup 
of index 2 in $\Gamma_3$. Using this, it is not so difficult to see that $\Gamma_3$ is a crystallographic group. Moreover,
$\Gamma_3$ is torsion free and hence a Bieberbach group. We leave it to the reader to check that 
$\Gamma_3\backslash \R^2$ is the Klein bottle.
\end{enumerate}

\medskip

The structure of crystallographic groups is well understood by the Bieberbach theorems\index{Bieberbach Theorems} (\cite{bieb11-1}, \cite{bieb12-1}, \cite{frob11-1}), which we recall here

\begin{Thm}[\bf First Bieberbach Theorem] 
Let $\Gamma\subseteq \Isom(\R^n)$ be a $n$-dimensional crystallographic group, then the group of translations 
$T=\Gamma\cap \R^n$ of $\Gamma$ is a lattice in $\R^n$ and is of finite index in $\Gamma$.
\end{Thm}

By a lattice of $\R^n$ we mean a cocompact and discrete subgroup of $\R^n$. Lattices of $\R^n$ are isomorphic to $\Z^n$ and each minimal generating set of a lattice of $\R^n$ is in fact a basis of the vector space $\R^n$.

\begin{Def}
Let $\Gamma\subseteq \Isom(\R^n)$ be a $n$-dimensional crystallographic group, then its linear part $r(\Gamma)$ is called the holonomy group\index{holonomy group} of $\Gamma$.
\end{Def}
The term holonomy is well chosen, because when $\Gamma$ is a Bieberbach group, the holonomy group  of $\Gamma$ is really the holonomy group of the flat manifold $M=\Gamma\backslash \R^n$. This fact will not be used further in this text, but the interested reader can consult \cite{char86-1} or \cite{wolf77-1} for more details.

\medskip

The first Bieberbach theorem states that a $n$-dimensional crystallographic group $\Gamma$ fits is a short exact 
sequence
\[ 0 \to t(\Gamma)\cong \Z^n \to \Gamma \to r(\Gamma)=F \to 1,\]
where $F$ is a finite group. Of course, being isomorphic to $\Z^n$, the translation group $t(\Gamma)$ is abelian. It is not so difficult to show that $t(\Gamma)$ is maximal abelian in $\Gamma$. 

\medskip

\begin{Thm}[\bf Second Bieberbach Theorem]
Let $\varphi:\Gamma \to \Gamma'$ be an isomorphism between two crystallographic groups, then $\Gamma$ and $\Gamma'$ are of the same dimension, say $n$, and moreover, there exists an affine map $\alpha\in\Aff(\R^n)$ such that 
\[ \forall \gamma \in \Gamma:\; \varphi(\gamma) = \alpha \circ \gamma \circ \alpha^{-1}.\]
\end{Thm}
Note that $\alpha \circ \gamma \circ \alpha^{-1}$ is a conjugation inside $\Aff(\R^n)$. This makes sense because both 
$\Gamma$ and $\Gamma'$ are subsets of $\Isom(\R^n)$ and so also of $\Aff(\R^n)$.

\medskip

Geometrically, this theorem says that a compact flat manifold is, up to an affine equivalence, completely determined by its fundamental group. The last Bieberbach theorem then states that, for a fixed $n$, there are up to affine equivalence only finitely many $n$-dimensional compact flat manifolds.

\begin{Thm}[\bf Third Bieberbach Theorem]
For any positive integer $n$, there are,
up to isomorphism (or up to affine conjugation), only finitely many $n$-dimensional crystallographic groups.
\end{Thm}

In \cite{zass48-1}, Zassenhaus proved a converse to the first Bieberbach theorem, which results in a complete algebraic characterization of the $n$-dimensional crystallographic groups:
\begin{Thm}[\bf Algebraic characterization of crystallographic groups]
Let $\Gamma$ be a $n$-dimensional crystallographic group, then $\Gamma$ fits in a short exact sequence
\[ 0 \to \Z^n \stackrel{i}{\to} \Gamma \to F \to 1, \]
where $F$ is finite and $i(\Z^n)$ is maximal abelian in $\Gamma$. Conversely, if $\Gamma$ is an abstract 
group fitting in a short exact sequence as above, where $F$ is finite and $i(\Z^n)$ is maximal abelian in $\Gamma$, then 
there exists an embedding $j:\Gamma \to \Isom(\R^n)$ such that $j(\Gamma)$ is a $n$-dimensional crystallographic group.
\end{Thm}

A little bit more can be said: when $\Gamma$ is a crystallographic group, then $t(\Gamma)$ is the unique normal 
abelian subgroup of $\Gamma$ which is also maximal abelian. Hence, as $t(\Gamma)$ is completely determined by algebraic properties, $t(\Gamma)$ is a characteristic subgroup of $\Gamma$.

By choosing a free generating set of the free abelian group $t(\Gamma)\cong \Z^n$, the short exact sequence
$0\to \Z^n \to \Gamma \to F\to 1$ (where $F=r(\Gamma)$) induces a representation 
$\varphi:F \to \Aut(\Z^n)$ by conjugation in $\Gamma$. This representation is often referred to as the holonomy representation\index{holonomy representation} of $\Gamma$. As $\Aut(\Z^n)\subseteq \GL(\Q^n)\subseteq\GL(\R^n)$ we can also via $\varphi$ as 
a representation into $\GL(\Q^n)$ (we then talk about the rational holonomy representation\index{rational holonomy representation}\label{ratholon}) or into $\GL(\R^n)$ (the real 
holonomy representation\index{real holonomy representation}). Note that also $r:\Gamma \to O(n) \subseteq \GL(\R^n)$ induces a representation 
$\bar{r}: \Gamma/t(\Gamma)=F \to \GL(\R^n)$. One can see that $\bar{r}$ and $\varphi$ are conjugate inside $\GL(\R^n)$.

\medskip

As a flat manifold $M$ is, up to affine equivalence, completely determined by its fundamental group, a Bieberbach group $\Gamma$, it should not come as a surprise that much of the geometry (or topology) of the manifold $M$ can be studied 
on the algebraic level. Very often, the holonomy representation plays a crucial role. We refer the reader to the book \cite{szcz12-1} to see quite some examples of this. In later sections of this article we will see that this is also the case 
for infra-nilmanifolds.

\section{Nilpotent Groups}\label{voorbeeldsec}
When moving from the class of flat manifolds (crystallographic groups) to the class of infra-nilmanifolds (almost crystallographic groups), we will  replace the abelian groups $\Z^n$ and $\R^n$ by their nilpotent analogues. It is therefore useful to spend some time describing these nilpotent groups, as in the previous section, this material is standard and we omit all proofs. For discrete nilpotent groups we refer the reader to 
\cite{hall69-1, km79-1,  khuk93-1, sega83-1} for more details, while for Lie group aspects  the references \cite{ov93-1,ov00-1,ov94-1} are a good starting point for more information. The book \cite{sega83-1} contains a complete description, with full proofs of everything we will say about rational Mal'cev completions.

\medskip

Let $G$ be any group, then we define the terms of the upper central series\index{upper central series} $Z_i(G)\subseteq G$ ($i\geq 0$) and those of the 
lower central series\index{lower central series} $\gamma_i(G)\subseteq G$ ($i\geq 1$) inductively as follows:
\[ Z_0(G) =1,\;\; Z_{i+1}(G) \mbox{ is determined by } Z_{i+1}(G)/Z_i(G)= Z( G/Z_i(G)),\]
\[ \gamma_1(G)= G,\;\; \gamma_{i+1}(G)= [ G, \gamma_i(G)] .\]
Note that a group is abelian $\Leftrightarrow Z_1(G)=G \Leftrightarrow \gamma_2(G) = 1$. 
As a generalization of abelian groups, we can consider groups for which the upper central series eventually stabilizes at $G$ 
or where the lower central series eventually stabilizes at the trivial group. The following lemma shows that both points 
of view lead to the same class of groups.
\begin{Lem}
Let $G$ be a group, then  for any positive integer $c$ we have that 
\[ Z_c(G)=G  \Leftrightarrow \gamma_{c+1}(G)=1.\]
\end{Lem}

\begin{Def}
A group $G$ is nilpotent\index{nilpotent group} of class $\leq c$ if and only if $Z_c(G)=G$ (if and only if $\gamma_{c+1}(G)=1$). We say that 
$G$ is nilpotent of class $c$ if $c$ is the smallest positive integer for which $Z_c(G)=G$ ($\gamma_{c+1}(G)=1$). 
\end{Def}
It follows that the abelian groups are exactly the nilpotent groups of class 1.

\begin{Ex} Let
\[ H=\left\{ \left( \begin{array}{ccc}
1 & y & z \\
0 & 1 & x \\
0& 0 & 1
\end{array} \right) \;|\; x,y,z \in \Z\right\}\]
be the group of upper triangular $3\times 3$ integral matrices. Then, $H$ is a nilpotent group of class $2$. The group $H$ is often referred to as the (integral) Heisenberg group.
More generally, for any commutative ring $R$ with 1, the group of upper triangular $n\times n$ matrices with entries in $R$ and 1's everywhere on the diagonal, 
say $UT_n(R)$,  is a nilpotent group of class $\leq n-1$. (E.g.\ see \cite[page 8]{sega83-1})
\end{Ex}
It is easy to prove that subgroups and quotient groups of nilpotent groups are also nilpotent (of at most the same class as the original group). 

\medskip

In the setting of crystallographic groups, the abelian groups $\Z^n$ play a crucial role, since any crystallographic group contains such a group as a characteristic subgroup of finite index. The groups $\Z^n$ are exactly the finitely generated and torsion free abelian groups. In the generalized setting of almost--crystallographic groups, the role of those $\Z^n$ will be replaced by the finitely generated and torsion free nilpotent groups. 

\medskip

So, since our aim is to discuss infra-nilmanifolds and almost--crystallographic groups in detail, it is necessary to obtain a good understanding of those finitely generated torsion free nilpotent groups.

\begin{Lem} Let $N$ be a nilpotent group, then
\begin{enumerate}
\item If $N$ is finitely generated, then any subgroup of $N$ is finitely generated (since $N$ is a polycyclic group).
\item If $N$ is torsion free, then for any positive integer $i$ also  $N/Z_i(N)$ is torsion free
\end{enumerate}
\end{Lem}

Note that for a torsion free nilpotent group $N$, it is not necessarily true that $N/\gamma_{i}(N)$ is torsion free.
E.g.\ for any positive integer $n$ we define 
\[ H_n=\left\{ \left( \begin{array}{ccc} 1 & y & \frac{z}{n} \\ 0 & 1 & x \\ 0 & 0 & 1 \end{array} \right) \;
|\; x,y,z \in \Z\right\}.\label{heisn}\]
Then $H_n$ is a torsion free nilpotent group (where $H_1$ is the Heisenberg group $H$ which we already met in the 
example above).
One can compute that 
\[ Z(H_n)=\left\{ \left( \begin{array}{ccc} 1 & 0 & \frac{z}{n} \\ 0 & 1 & 0 \\ 0 & 0 & 1 \end{array} \right) \;
|\; z \in \Z\right\} \mbox{ and }\gamma_2(H_n)=\left\{ \left( \begin{array}{ccc} 1 & 0 & z\\ 0 & 1 & 0 \\ 0 & 0 & 1 \end{array} \right) \;
|\; z \in \Z\right\},\]
from which is follows that 
\[ H_n/Z(H_n)\cong \Z^2 \mbox{ is indeed torsion free, while }H_n/\gamma_2(H_n)=\Z^2\oplus \Z_n\mbox{ is not, when $n\neq 1$.}\]
It follows from the lemma above that for a given finitely generated torsion free nilpotent group $N$ of class $\leq c$, 
we have that 
\[ N =Z_c(N) \supseteq Z_{c-1}(N)  \supseteq \cdots \supseteq Z_2(N) \supseteq Z_{1}(N)\supseteq 1=Z_0(N)\]
is a filtration of $N$ with $Z_{i}(N)/Z_{i-1}(N) \cong \Z^{k_i}$ for some integers $k_i$. 
By refining this filtration (and numbering in the other direction), 
we can find a series of normal subgroups $N_i$ of $N$:
\[ N=N_1\supset N_2 \supset N_3 \cdots \supset N_k\supset N_{k+1}=1\]
with $N_i/N_{i+1}\cong \Z$ and $N_i/N_{i+1} \subseteq Z(N/N_{i+1})$ (which is equivalent to 
$[N,N_i]\subseteq N_{i+1}$).
Now, given a series of subgroups $N_i$ of $N$ satisfying these conditions, we 
fix an element $a_i\in N_i$ for all $i$, in such a way that the natural projection of $a_i$ in $N_i/N_{i+1}$ is a generator of the infinite cyclic group $N_i/N_{i+1}$. It is then obvious to see that 
any element $x\in N$ can be uniquely expressed as an element of the form 
\[ x = a_1^{z_1} a_2^{z_2}\ldots a_k^{z_k} \mbox{ for some }z_1,z_2, \ldots, z_k\in \Z.\] 
We will refer to $a_1, a_2, \ldots , a_k$ as a Mal'cev basis\index{Mal'cev basis} for $N$. 
\begin{Prop}\label{polprod}
Let $N$ be a finitely generated torsion free nilpotent group with Mal'cev basis $a_1,a_2,\ldots,a_k$, then there 
exist polynomials $p_i(x_1,x_2,\ldots,x_{i-1},y_1,y_2,\ldots y_{i-1})$ ($2\leq i \leq k$) with coefficients in $\Q$ and 
polynomials $q_i(x_1,x_2,\ldots,x_{i-1},z)$  ($2\leq i \leq k$) also with coefficients in $\Q$ such that 
for all $x_1,x_2, \ldots x_k,y_1,\ldots y_k,z\in \Z$ we have that 
\begin{enumerate}
\item 
$a_1^{x_1}a_2^{x_2}\ldots a_k^{x_k} a_1^{y_1}a_2^{y_2}\ldots a_k^{y_k}=$
\[ a_1^{x_1+y_1} a_2 ^{x_2+y_2+p_2(x_1,y_1)}\ldots 
a_i^{x_i +y_i+ p_i(x_1,x_2,\ldots,x_{i-1},y_1,y_2,\ldots y_{i-1})} \ldots a_k^{x_k+y_k+p_k(x_1, \ldots , y_{k-1})}.\]
\item $\left( a_1^{x_1}a_2^{x_2}\ldots a_k^{x_k} \right)^z =$
\[ a_1^{z x_1} a_2^{z x_2 +q_1(x_1,z)} \ldots a_i^{z x_i + q_i(x_1,x_2,\ldots, x_{i-1},z)} \ldots 
a_k^{z x_k + q_k(x_1,\ldots ,x_{k-1}, z)}.\]
\end{enumerate}
\end{Prop}

Using these polynomials, we can now construct two new groups, $N^\Q$ and $N^\R$. These groups consist formally 
of the elements $a_1^{x_1}a_2^{x_2}\ldots a_k^{x_k}$ where the $x_i\in \Q$ (respectively $\R$) and where the 
product rule is given by means of the same polynomials $p_i$. Both $N^\Q$ and $N^\R$ are torsion free nilpotent groups (but of course, they are no longer finitely generated).

The group $N^\Q$ is a radicable group, this is a group such that any element of $N^\Q$ has a (unique) $m$--th root in $N^\Q$ (for all positive integers $m$). Moreover, $N^\Q$ is the radicable closure of $N$, i.e.\ it is a radicable group containing $N$ and any element of $N^\Q$ has a positive power lying in $N$. This implies that $N^\Q$ is the smallest radicable group containing $N$. This property does determine $N^\Q$ uniquely up to isomorphism. Often $N^\Q$ is referred to as the rational Mal'cev completion\index{rational Mal'cev completion} of $N$.

\begin{Ex}\label{ratmalcev}
It is easy to see that when $N\cong \Z^k$, then $N^\Q\cong \Q^k$. For the integral 
Heisenberg group we have that 
\[ H^\Q = \left\{ \left( \begin{array}{ccc}
1 & y & z \\
0 & 1 & x \\
0& 0 & 1
\end{array} \right) \;|\; x,y,z \in \Q\right\}\]
In fact, for every $H_n$ we have that $H_n^\Q\cong H^\Q$.
\end{Ex}

\begin{Prop}
Let $\varphi: N_1 \to N_2$ be a homomorphism between two finitely generated torsion free nilpotent groups. Then 
$\varphi$ uniquely extends to a homomorphism $\varphi^\Q:N_1^\Q \to N_2^\Q$. Moreover, when $\varphi$ is an isomorphism,
so is $\varphi^\Q$.
\end{Prop}
The proof of the proposition is not difficult, since it is easy to see that the homomorphism $\varphi$ is expressed as a polynomial map in terms of coordinates with respect to chosen Mal'cev bases of $N_1$ and $N_2$. Using the same polynomial map, one then defines the homomorphism between $N_1^\Q$ and $N_2^\Q$  (but now allowing also rational coordinates). 

\medskip

In the example \ref{ratmalcev} above we saw that it is  possible that $N_1\not \cong N_2$, but $N_1^\Q \cong N_2^\Q$. This occurs exactly when the groups $N_1$ and $N_2$ are commensurable.
\begin{Def}
Two groups $N_1$ and $N_2$ are said to be (abstractly) commensurable if there exists finite index subgroups
subgroups $H_1\leq N_1$ and $H_2 \leq N_2$ such that $H_1\cong H_2$.
\end{Def}

\begin{Prop}
let $N_1$ and $N_2$ be two finitely generated torsion free nilpotent groups, then 
\[ N_1^\Q \cong N_2^\Q \Leftrightarrow \mbox{$N_1$ and $N_2$ are commensurable.}\]
\end{Prop}

The group $N^\R$ is called the real Mal'cev completion\index{Mal'cev completion} of $N$. This is a Lie group. Indeed as a manifold $N^\R$ is diffeomorphic to $\R^k$ and the group operations are polynomial in the coordinates, so these are certainly smooth. 
The group $N^\R$ is a simply connected nilpotent Lie group which contains $N$ as a discrete cocompact subgroup.  
Generalizing 
the abelian case, we formulate the following definition.
\begin{Def}
 A lattice of a simply connected Lie group $G$ is a discrete and cocompact subgroup $N$ of $G$.
\end{Def} 
It is known that, up to a Lie group isomorphism, $N^\R$ is the unique connected and simply connected nilpotent Lie group 
containing $N$ as a lattice. At this point we want to remark that there exist nilpotent groups with different rational Mal'cev completions, but with the same real Mal'cev completion.

Analogously to the situation for rational Mal'cev completions, we have the following result:
\begin{Prop}
Let $\varphi: N_1 \to N_2$ be a homomorphism between two finitely generated torsion free nilpotent groups. Then 
$\varphi$ uniquely extends to a continuous homomorphism $\varphi^\R:N_1^\R \to N_2^\R$. Moreover, when $\varphi$ is an isomorphism,
so is $\varphi^\R$.
\end{Prop}

To study radicable nilpotent groups, like $N^\Q$, or nilpotent Lie groups, like $N^\R$, it is often useful to move over 
to their associated Lie algebras. The rest of this section is devoted to the description of this 
connection in our situation. To explain this, it is easy to make use of the following result: 

\begin{Thm} Let $N$ be a finitely generated torsion free nilpotent group. Then there exists an embedding 
$\varphi:N \to UT_n(\Z)$,  for some $n$. So any finitely generated 
torsion free nilpotent group can be realized as a subgroup of some $UT_n(\Z)$ and any subgroup 
of $UT_n(\Z)$ is a finitely generated torsion free nilpotent group.
\end{Thm}
From now onwards, we will assume that 
\[N\subseteq UT_n(\Z) \subseteq UT_n(\Q) \subseteq UT_n(\R).\]
Let us now use $NT_n(R)$ to denote the set of $n\times n$ upper triangular matrices with entries in the ring $R$ and 
with a 0 everywhere on the diagonal. So all elements of $NT_n(R)$ are nilpotent matrices.
By considering the usual  Lie bracket $[A,B]=AB-BA$ for matrices, one easily sees that $NT_n(\Q)$ (resp.\ $NT_n(\R)$)
is a Lie algebra over $\Q$ (resp.\ over $\R$).

\medskip

We introduce two maps:
\[ \exp: NT_n(\R^n) \to UT_n(\R^n): A \mapsto \sum_{k=0}^\infty \frac{1}{k!} A^k \]
and 
\[ \log: UT_n(\R^n) \to NT_n(\R^n): A \mapsto \sum_{k=1}^ \infty \frac{(-1)^{k+1}}{k} (A-I_n)^k\]
where $I_n$ is the $n\times n$ identity matrix.\\
Note that in the definitions above, we have that the matrices $A$ and $A-I_n$ appearing in the expression for $\exp$ and
$\log$ respectively, are nilpotent matrices and so the infinite sums are actually finite sums.
The maps $\exp$ and $\log$ are each others inverse and so are both bijective.

The Campbell--Baker--Hausdorff formula\index{Campbell--Baker--Hausdorff formula} provides a relation between the multiplication in the group $UT_n(\R)$ and 
the Lie bracket in the Lie algebra $NT_n(\R)$.
This formula says that $\forall X,Y\in NT_n(\R): $
\[ \exp(X) \exp(Y) = \exp\left( 
X+Y + \frac{1}{2} [X,Y]+ \frac{1}{12} ([X,[X,Y]] + [Y,[Y,X]] ) + \sum_{k=4}^\infty r_k \right)\]
where each $r_k$ is a rational combination of $k$--fold Lie brackets in $X$ and $Y$. Again the sum is in fact a finite sum,
since $k$-fold brackets are 0 when $k$ is big enough.

\begin{Thm}
Let $N\subseteq UT_n(\Z)$ and take 
\[ \lien^\Q= \Q \log(N) \mbox{ (the $\Q$--vector space generated by $\log(N)$) and}\]
\[ \lien^\R= \R \log(N) \mbox{ (the $\R$--vector space generated by $\log(N)$).}\]
Then
\begin{enumerate}
\item $\lien^\Q$ is a rational Lie algebra and $N^\Q=\exp(\lien^\Q)$.
\item $\lien^\R$ is a real Lie algebra and $N^\R=\exp(\lien^\R)$.
\end{enumerate}
\end{Thm}
We remark here that the isomorphism type of the Lie algebras $\lien^\Q$ and $\lien^\R$ are independent of the chosen embedding of $N$ into $UT_n(\Z)$ (even into some $UT_n(\Q)$ and this also for various $n$) and so is really an invariant of the group $N$.

This correspondence behaves nicely on the level of (iso)morphisms:
\begin{Thm}\label{differential}
Let $N_1$ and $N_2$ be two finitely generated torsion free nilpotent groups and let $\F=\Q$ or $\R$.
Then, for any group homomorphism $\varphi:N_1^\F \to N_2^\F$ (continuous in the case $\F=\R$) there exists 
a unique Lie algebra homomorphism $\varphi_\ast: \lien_1^\F \to \lien_1^\F$ making the following diagram  commutative:
\[ 
\xymatrix{ N_1^\F \ar[r]^\varphi\ar@<2pt>[d]^\log & N_2^\F\ar@<2pt>[d]^\log\\
\lien_1^\F\ar[r]_{\varphi_\ast}\ar@<2pt>[u]^\exp & \lien_2^\F\ar@<2pt>[u]^\exp }
\]
Conversely, given a Lie algebra homomorphism $\varphi_\ast: \lien_1^\F \to \lien_1^\F$, there is a unique 
(continuous) homomorphism $\varphi:N_1^\F \to N_2^\F$ making the diagram commutative.\\
Under this correspondence, we have that $\varphi$ is bijective (resp. injective, surjective) if and only if $\varphi_\ast$ is.
\end{Thm} 

So as a corollary we find that 
\begin{Cor}
Let $N$ be a finitely generated torsion free nilpotent group, then 
\[ \Aut(N^\Q)=\Aut(\lien^\Q)\mbox{ \ and \ } \Aut(N^\R) = \Aut(\lien^\R).\]
\end{Cor}

Before we continue let us illustrate all of this by means of a rather elaborate example.

Let us consider the Heisenberg group $H$ again and let 
\[\label{abc} a=\left( \begin{array}{ccc} 1 & 0 & 0 \\
                                                          0 & 1 & 1 \\
                                                          0 &  0 & 1 \end{array} \right), \;\;
b=\left( \begin{array}{ccc} 1 & 1 & 0 \\
                                                          0 & 1 & 0 \\
                                                          0 &  0 & 1 \end{array} \right)\mbox{ and }                                          
c=\left( \begin{array}{ccc} 1 & 0 & 1 \\
                                                          0 & 1 & 0 \\
                                                          0 &  0 & 1 \end{array} \right).\]
Then $a,\;b,\;c$ form a Mal'cev basis of $H$ and one computes that 
\[ a^xb^yc^z= \left(              \begin{array}{ccc} 1 & y & z \\
                                                          0 & 1 & x \\
                                                          0 &  0 & 1 \end{array} \right)    .\]
The group $H$ has a presentation of the form
\[ H=\langle a,b,c \;|\; [b,a]=c,\; [c,a]=[c,b]=1 \rangle,\]
where we define the commutator as $[x,y]=x^{-1} y^{-1} x y$.

\medskip
                                                          
It is easy to check that 
\[\forall x_1,x_2,y_1,y_2,z_1,z_2\in \Z:\; a^{x_1} b^{y_1} c^{z_1} a^{x_2} b^{y_2} c^{z_2} = a^{x_1+x_2} b^{y_1+y_2} c^{z_1+z_2 + x_2 y_1}\] 
and
\[ \forall x,y,z,m\in \Z:\; \left(a^xb^yc^z\right)^m= a^{mx}b^{my}c^{mz + \frac{m(m-1)}{2} xy },\] 
which illustrates Proposition~\ref{polprod}.

\medskip

Now we compute the rational Lie algebra associated to $H$ (the case for the real Lie algebra is completely analogous).
We have that 
\[ \log(a^xb^yc^z) = \log \left(              \begin{array}{ccc} 1 & y & z \\
                                                          0 & 1 & x \\
                                                          0 &  0 & 1 \end{array} \right)=
               \left(              \begin{array}{ccc} 0 & y & z - \frac{xy}{2} \\
                                                          0 & 0 & x \\
                                                          0 &  0 & 0 \end{array} \right).\]
It follows that 
\[ \lieh^\Q = \Q\log(H) = \left\{ \left(   \begin{array}{ccc} 0 & y & z - \frac{xy}{2} \\
                                                          0 & 0 & x \\
                                                          0 &  0 & 0 \end{array} \right)|\, x,y,z\in \Q \right\}=
                                                           \left\{ \left(   \begin{array}{ccc} 0 & y & z  \\
                                                          0 & 0 & x \\
                                                          0 &  0 & 0 \end{array} \right)|\, x,y,z\in \Q \right\}=NT_3(\Q).\]
From this, it follows that $H^\Q=\exp(\lieh^\Q)=UT_3(\Q)$ as claimed already in Example~\ref{ratmalcev}.

\medskip

Now, consider 
\[ \label{auto3}\varphi: H^\Q \to H^\Q:  \left(  \begin{array}{ccc} 1 & y & z \\
 0 & 1 & x \\
 0 &  0 & 1 \end{array} \right) \mapsto 
\left(  \begin{array}{ccc} 1 & x-y & z+x -xy + \frac{y(y-1)}{2} \\
 0 & 1 & -y \\
 0 &  0 & 1 \end{array} \right).\]
One can check that $\varphi \in \Aut(H^\Q)$ (in fact $\varphi(H)=H$, so $\varphi$ can also be seen as an automorphism 
of $H$). Let us now compute the Lie algebra homomorphism $\varphi_\ast:\lieh^\Q\to \lieh^\Q$ from Theorem~\ref{differential}. 
It holds that $\forall X\in \lieh^\Q$ we have that $\varphi_\ast (X) = \log ( \varphi ( \exp(X)))$,
so 
\begin{eqnarray*}
\varphi_*\left(  \begin{array}{ccc} 0 & y & z  \\
                                                          0 & 0 & x \\
                                                          0 &  0 & 0 \end{array} \right) 
& = & \log \left( \varphi \left ( \begin{array}{ccc} 1 & y & z + \frac{xy}{2} \\
                                                          0 & 1 & x \\
                                                          0 &  0 & 1 \end{array} \right) \right) \\
& = & \log \left ( \begin{array}{ccc} 
1 & x-y & z+x -\frac{xy}{2} + \frac{y(y-1)}{2} \\
 0 & 1 & -y \\
 0 &  0 & 1
 \end{array} \right)  \\
 & = & \left(  \begin{array}{ccc} 0 & x-y & z + x - \frac{y}{2}  \\
                                                          0 & 0 & -y \\
                                                          0 &  0 & 0 \end{array} \right).
\end{eqnarray*}
From the above, it is obvious that $\varphi_\ast$ is a linear endomorphsim of $\lieh^\Q$. A small calculation also shows that 
$\varphi_\ast$ is in fact a Lie algebra automorphism. Another computation will show that $\varphi_\ast^3$ is the indentity map on
$\lieh^\Q$. From this it follows that also $\varphi^3$ is the identity map of $H^\Q$.

\medskip

In the sequel, it will be useful to work with another representation of $H$ (and $H^\Q$) as a group of unipotent matrices: 
one easily computes that the map 
\[ \label{reppsi}\psi=UT_3(\Q) \to UT_4(\Q): 
 \left(  \begin{array}{ccc} 1 & y & z \\
    0 & 1 & x \\
    0 &  0 & 1 \end{array} \right)   \mapsto
\left( \begin{array}{cccc}
1 & \frac{y}{2} & -\frac{x}{2} & z - \frac{xy}{2}\\
0 & 1 & 0 & x \\
0 & 0 & 1 & y \\
0 & 0 & 0 & 1 \end{array}\right) \]
is an injective homomorphism of groups.

\medskip
Let 
\[ N = \left\{ \left(\begin{array}{cccc}
1 & \frac{y}{2} & -\frac{x}{2} & z - \frac{xy}{2}\\
0 & 1 & 0 & x \\
0 & 0 & 1 & y \\
0 & 0 & 0 & 1 \end{array}\right)\;|\; x,y,z\in \Z \right\}.\]
Then $H\cong N$ (via $\psi$ above).  As 
\[ \log \left(\begin{array}{cccc}
1 & \frac{y}{2} & -\frac{x}{2} & z - \frac{xy}{2}\\
0 & 1 & 0 & x \\
0 & 0 & 1 & y \\
0 & 0 & 0 & 1 \end{array}\right) = \left(\begin{array}{cccc}
0 & \frac{y}{2} & -\frac{x}{2} & z - \frac{xy}{2}\\
0 & 0 & 0 & x \\
0 & 0 & 0 & y \\
0 & 0 & 0 & 0 \end{array}\right)\]
we find that 
\[ \lien^\Q= \left\{  \left(\begin{array}{cccc}
0 & \frac{y}{2} & -\frac{x}{2} & z - \frac{xy}{2}\\
0 & 0 & 0 & x \\
0 & 0 & 0 & y \\
0 & 0 & 0 & 0 \end{array}\right)\;|\; x,y,z\in \Q \right\}= \left\{  \left(\begin{array}{cccc}
0 & \frac{y}{2} & -\frac{x}{2} & z \\
0 & 0 & 0 & x \\
0 & 0 & 0 & y \\
0 & 0 & 0 & 0 \end{array}\right)\;|\; x,y,z\in \Q \right\}
.\]
It follows that 
\[ N^\Q=  \left\{ \exp \left(\begin{array}{cccc}
0 & \frac{y}{2} & -\frac{x}{2} & z \\
0 & 0 & 0 & x \\
0 & 0 & 0 & y \\
0 & 0 & 0 & 0 \end{array}\right)\;|\; x,y,z\in \Q \right\}=
 \left\{  \left(\begin{array}{cccc}
1 & \frac{y}{2} & -\frac{x}{2} & z \\
0 & 1 & 0 & x \\
0 & 0 & 1 & y \\
0 & 0 & 0 & 1 \end{array}\right)\;|\; x,y,z\in \Q \right\},\]
so we have that $\psi(UT_3(\Q))=\psi(H^\Q)= N^\Q$.\\
Note that 
\[ \psi_\ast: NT_3(\Q) \to NT_4(\Q):
  \left(  \begin{array}{ccc} 0 & y & z \\
    0 & 0 & x \\
    0 &  0 & 0 \end{array} \right)   \mapsto
\left( \begin{array}{cccc}
0 & \frac{y}{2} & -\frac{x}{2} & z \\
0 & 0 & 0 & x \\
0 & 0 & 0 & y \\
0 & 0 & 0 & 0 \end{array}\right)\]
is an injective homomorphism of Lie algebras. In fact $\psi_\ast$ is really the homomorphism making the diagram 
\[ 
\xymatrix{ H^\Q \ar[r]^\psi\ar@<2pt>[d]^\log & N^\Q\ar@<2pt>[d]^\log\\
\lieh^\Q\ar[r]_{\psi_\ast}\ar@<2pt>[u]^\exp & \lien^\Q\ar@<2pt>[u]^\exp }
\]
commutative as in Theorem~\ref{differential}. Note that this observation illustrates that the Lie algebra 
associated to a torsion free nilpotent group is independent of the way this group is represented as a subgroup 
of some $UT_n(\Q)$.

\medskip

We will now indicate a little bit why this second representation of the Heisenberg group as a group of $4 \times 4$ matrices
has some advantage. \label{voordeel}To see this, we consider
\[\label{defA} A= \left( \begin{array}{cccc} 
1 & 1 & -\frac12& 0 \\
0 & 0 & -1 & 0\\
0 & 1 & -1 & 0 \\
0 & 0 & 0 & 1
\end{array}\right). \]
Some computations show that 
\[ A^3=I_4\mbox{ and }
A.\left( \begin{array}{cccc}
1 & \frac{y}{2} & -\frac{x}{2} & z-\frac{xy}{2} \\
0 & 1 & 0 & x \\
0 & 0 & 1 & y \\
0 & 0 & 0 & 1 \end{array}\right).A^{-1} = 
\left( \begin{array}{cccc}
1 & \frac{x-y}{2} & \frac{y}{2} & z+x-\frac{y}{2}-\frac{xy}{2}  \\
0 & 1 & 0 & -y \\
0 & 0 & 1 & x-y \\
0 & 0 & 0 & 1 \end{array}\right)\in N^\Q,\]
from which we can conclude that conjugation with $A$ induces an automorphism of $N^\Q$ of order 3, let us call
this automorphism $\varphi_A$. To see what the corresponding automorphism on $\lien^\Q$ looks like,
let us first of all remark that 
\[ \varphi_{A\ast}(X)= \log ( \varphi_A(\exp (X))) = 
\log (A \exp(X) A^{-1} ) = \log \exp (AXA^{-1}) = AXA^{-1}\]
from which we obtain that also $\varphi_{A\ast}$ is given by conjugation with $A$. Hence we find:
\[
\varphi_{A\ast}  \left( \begin{array}{cccc}
0 & \frac{y}{2} & -\frac{x}{2} & z \\
0 & 0 & 0 & x \\
0 & 0 & 0 & y \\
0 & 0 & 0 & 0 \end{array}\right) = 
 \left( \begin{array}{cccc}
0 & \frac{x-y}{2} & \frac{y}{2} & z+x-\frac{y}{2} \\
0 & 0 & 0 & -y \\
0 & 0 & 0 & x-y  \\
0 & 0 & 0 & 0 \end{array}\right) .\]
It follows that there is a commutative diagram of the form 
\[\xymatrix{ \lieh^\Q\ar[r]^{\varphi_\ast}\ar[d]_{\psi_\ast} & \lieh^\Q\ar[d]^{\psi_\ast}\\
\lien^\Q\ar[r]_{\varphi_{A\ast}} & \lien^\Q },\]
so this means that under the identification of $H^\Q$ with $N^\Q$ via $\psi$, we have that 
$\varphi$ corresponds to the automorphism $\varphi_A$. 
The advantage of viewing the rational Heisenberg group as being the group $N^\Q$ is that any automorphism 
of $N^\Q$ can be described by means of a conjugation by a matrix (of a nice form). We will come back to this fact later in 
section~\ref{computation}.
\section{Infra-nilmanifolds and almost--crystallographic groups}
We are now ready to introduce the main concepts of this text, namely the almost--crystallographic groups and the infra-nilmanifolds. In the setting of the usual crystallographic groups, the basic space was the abelian Lie group 
$\R^n$. Now we generalize this and we consider a connected and simply connected nilpotent Lie group $G$. 
Let $\Aut(G)$ be the group of Lie automorphisms of $G$. Then we can consider the semidirect product
$G\semi \Aut(G)$, where the group operation is given by 
\[ \forall g,h \in G ,\; \forall \alpha, \beta \in\Aut(G):\; (g,\alpha) ( h,\beta) = (g \alpha(h), \alpha \beta).\]
We will refer to this group as the affine group of $G$\index{Affine group of a Lie group} and we denote it by $\Aff(G)$. This affine 
group acts on $G$ in the following way:

\[ \forall (g,\alpha) \in \Aff(G),\; \forall x \in G:\; (g,\alpha)\cdot  x = g \alpha(x).\]
Note that, when $G=\R^n$, all of this agrees with what we earlier introduced for $\Aff(\R^n)$. When,
introducing the crystallographic groups, we regarded $\R^n$ as being a Euclidean space and then we obtained the 
group of isometries $\Isom(\R^n) = \R^n\semi O(n)$ as a subgroup of $\Aff(\R^n)$. 
For general $G$, it is perhaps not so obvious at first what should be the analogue of the Euclidean structure and the corresponding isometry group. To see this, let us focus on the group $O(n)$ (of transformations of the vector space $\R^n$ preserving the inner product).
This group $O(n)$ is a maximal compact subgroup of $\Aut(\R^n)$ and is in this sense uniquely determined up to 
inner conjugation inside $\Aut(\R^n)$. Therefore, when moving on to the setting of almost--crystallographic groups, we fill fix a maximal compact subgroup $C\subseteq \Aut(G)$. Again, this $C$ is determined up to conjugation inside $\Aut(G)$. Although we will not need this in the rest of the text, we remark here that the group $G\semi C$ can also be interpreted as a group of isometries of $G$ in the following way:
We have seen that the group of automorphisms of $G$ is isomorphic to the group of automorphisms 
of its Lie algebra $\lie$. So we can consider $C$ as being a maximal compact subgroup of $\Aut(\lie)$ and then
there exists a positive definite inner product on $\lie$ which is invariant under the action of $G$ (see \cite{vinb89-1}).

The Lie algebra of $G$ can be interpreted as being the tangent space of $G$ at the identity element 1.
Now, we define a metric on $G$ by requiring that this metric is left invariant and that it coincides with the 
inner product at the tangent space $T_1G=\lie$ where the chosen inner product is invariant under the action of $C$.
For this metric, we then have that $G\semi C$ is a group of isometries of $G$.

\begin{Def}
An almost--crystallographic group\index{almost--crystallographic group} (modeled on $G$), is a cocompact and discrete subgroup of $G\semi C$, where $G$ and $C$ are as introduced above.\\
The dimension of an almost--crystallographic group modeled on $G$ is the dimension of $G$.\\
An almost--Bieberbach group\index{almost--Bieberbach group} is a torsion free almost--crystallographic group.
\end{Def}
 
At the end of this section, we give an explicit  example of an almost--crystallographic group (which is not torsion free) and 
of an almost--Bieberbach group.  The interested reader can already look at these examples now and then read the rest of this section.
 
\begin{Prop}
Let $\Gamma$ be an almost--crystallographic group modeled on $G$. Then
\begin{enumerate}
\item $\Gamma$ (as a subgroup of $\Aff(G)$) acts properly discontinuously and cocompactly on $G$.
\item The action of $\Gamma$ on $G$ is free $\Leftrightarrow$ $\Gamma$ is an almost--Bieberbach group.
\end{enumerate}
\end{Prop}

\begin{Def}
An infra--nilmanifold\index{infra--nilmanifold} is a quotient space $\Gamma\backslash G$, where $\Gamma$ is an almost--Bieberbach group modeled on $G$. If moreover,  $\Gamma$ is completely contained inside $G$ (so none of the elements in $\Gamma$ have a non-trivial $C$-component), then $\Gamma\backslash G$ is called a nilmanifold.
\end{Def}
So, the class of infra--nilmanifolds is a generalization of the class of flat manifolds and a nilmanifold\index{nilmanifold} is in this sense a generalization of a torus.

\medskip

All three Bieberbach theorems have been generalized to the class of infra-nilmanifols. The first one has a 
straightforward generalization and was proved already in 1960 by L.~Auslander (\cite{ausl60-1}):
\begin{Thm}[\bf Generalized First Bieberbach Theorem]\label{bieb1}
Let $\Gamma \subseteq G\semi C$ be an almost--crystallographic group modeled on $G$, then 
$N=\Gamma \cap G$ is a lattice of $G$ and $\Gamma/N$ is finite.
\end{Thm}
We will refer to the group $N=\Gamma\cap G$ as the the group of pure translations of $\Gamma$.

\medskip

Consider the natural homomorphism $p:\Aff(G)=G \semi \Aut(G) \to \Aut(G)$. Let $\Gamma\subseteq G\semi G$ 
be an almost--crystallographic group. Then the first generalized first Bieberbach  says that the kernel 
of the restriction of $p$ to $\Gamma$, namely $G\cap \Gamma$ is a lattice in $G$ and the 
image $p(\Gamma)$ is a finite group. We will call this finite group $F=p(\Gamma)$ the holonomy group\index{holonomy group} of 
$\Gamma$.

\medskip

A useful fact when working with  with almost--crystallographic groups is the fact that the group of pure 
translations of an almost--crystallographic group $\Gamma$ is maximal nilpotent in $\Gamma$. 
Since I was unable to locate a complete proof of this fact in the literature, I will present a proof here.  

For this we will also need the following group theoretical lemma

\begin{Lem}
Let $C_n =\langle t \rangle$ be a cyclic group of finite order $n$ and let $E$ be an extension 
\[ 1 \to \Z^k \to E \to C_n \to 1.\]
Then the following are equivalent 
\begin{enumerate}
\item $E$ is a nilpotent group.
\item $E$ is an abelian group.
\item The action of $C_n$ on $\Z^k$ induced by conjugation in $E$ is the trivial action. 
\end{enumerate}
\end{Lem}
\begin{proof}
The implications $(3) \Rightarrow (2)$ and $(2) \Rightarrow (1)$ are obvious, so we just have to 
prove the implication $(1) \Rightarrow (3)$. The action of the generator $t$ of $C_n$ on $\Z^k$ is given 
by a $k\times k$ integral matrix $A$. We will show by induction on $k$ that $A$ has to be a unipotent 
matrix. The case $k=0$ is trivially true. As $E$ is nilpotent, $Z(E)\cap \Z^k$ is non trivial, 
since in any nilpotent group it holds that 
the intersection of a non-trivial normal subgroup with the center is non-trivial. Let $Z=Z(E) \cap \Z^k$.
It is not so difficult to see that $\Z^k/Z$ is torsion free, hence $\Z^k/Z\cong \Z^l$ with $l<k$. Also the group 
$E/Z$ is nilpotent and fits in an extension $1\to \Z^l \to E/Z \to C_n \to 1$. By induction, we find that the 
action on $\Z^l$ is unipotent. Now, we have an extension of $C_n$--modules
$0 \to Z \to Z^k \to \Z^l \to 0 $, where $C_n$ acts trivially on $Z$ and unipotently on $\Z^l$. It follows that 
$C_n$ acts unipotently on $\Z^k$, hence $A$ is a unipotent matrix. But as $t$ is of order $n$, we must have that 
$A$ is also of finite order. The only unipotent matrix of finite order is the identity matrix, so $C_n$ acts 
trivially on $\Z^k$.
\end{proof}

\begin{Prop} \label{maxnilp}Let $\Gamma$ be an almost--crystallographic group modeled on $G$.\\
Then, the group of pure translations $N=\Gamma\cap G$ is maximal nilpotent in $\Gamma$.
\end{Prop}
\begin{proof}
Let us start with a remark: if $\lie$ is a nilpotent Lie algebra and $\alpha\in \Aut(\lie)$ is such that 
$\alpha$ induces the identity automorphism on $\lie/\gamma_2(\lie)$, then by induction on $i$ one gets that 
 $\alpha$ will also induce the identity automorphism on all quotients $\gamma_{i}(\lie)/\gamma_{i+1}(\lie)$ and 
hence $\alpha$ is a unipotent automorphism of $\lie$. If moreover $\alpha$ is of finite order, this will imply that 
$\alpha$ is the trivial automorphism of $\lie$. From this it also follows that for a given simply connected 
nilpotent Lie group $G$ and an automorphism of finite order $\alpha\in \Aut(G)$ inducing the 
identity on $G/\gamma_2(G)$, we have that $\alpha$ is itself the identity. 

\medskip

Now, consider $x\in \Gamma$, such that $\langle N, x \rangle$ is nilpotent. Such an element $x$ is of the 
form $x=(g,\alpha) \in G \semi C$, where $\alpha$ is of finite order $n$.
Now, consider the short exact sequence
\[ 1 \to \frac{N}{N\cap \gamma_2(G)} \to \frac{\langle N, x \rangle}{N\cap \gamma_2(G)} \to 
C_n\cong \frac{\langle N, x \rangle}{N} \to 1.\]
The group $\frac{N}{N\cap \gamma_2(G)}$ is torsion free and hence isomorphic to $\Z^k$ for some $k$.
By the lemma above, we know that the action of $x$ on this group has to be trivial. But this action 
is the same as the action which is induced by $\alpha$ on $\frac{N}{N\cap \gamma_2(G)}$. As 
$\frac{N}{N\cap \gamma_2(G)}$ is a lattice of $G/\gamma_2(G)$, it follows that $\alpha$ induces the identity 
automorphism on $G/\gamma_2(G)$ and by the remark above, we have that $\alpha$ is the identity automorphism.
Hence $x=(g,1)\in N$ and so $N$ is maximal nilpotent.
\end{proof}

\medskip

Also the second Bieberbach Theorem has a straightforward generalization. This theorem was proved by K.B.~Lee 
and F.~Raymond in \cite{lr85-1}.
\begin{Thm}[\bf Generalized Second Bieberbach Theorem]\label{tweedebieb}
Let $\varphi:\Gamma \to \Gamma'$ be an isomorphism between two almost--crystallographic groups, then $\Gamma$ and $\Gamma'$ are modeled on the same Lie group $G$, and moreover, there exists an affine map 
$\alpha\in\Aff(G)$ such that 
\[ \forall \gamma \in \Gamma:\; \varphi(\gamma) = \alpha \circ \gamma \circ \alpha^{-1}.\]
\end{Thm}
We remark here that in \cite{lr85-1} it was assumed from the beginning that $\Gamma$ and $\Gamma'$ were modeled on the same nilpotent Lie group. However assume that $\Gamma$ is modeled on $G$ and $\Gamma'$ 
is modeled on $G'$. Then, if $\varphi:\Gamma \to \Gamma'$ is an isomorphism, 
$\varphi$  also induces an isomorphism between $\Gamma\cap G$ and $\Gamma'\cap G'$ because these groups 
are both characterized as being the unique normal and maximal nilpotent subgroup of $\Gamma\cap G$ 
resp.~$\Gamma'\cap G'$. As these groups are lattices in $G$ and $G'$, this isomorphism extends uniquely to 
an isomorphism of $G$ to $G'$ and hence we can assume that $G$ equals $G'$. Note that for this theorem it is not necessary that the same maximal compact subgroup $C\subseteq \Aut(G)$ is used to define both $\Gamma$ and 
$\Gamma'$.

In a following section, we will even further generalize this theorem for general homomorphisms $\varphi$. This will 
lead to a complete understanding of all maps between two given infra-nilmanifolds up to homotopy.

\medskip

The third Bieberbach Theorem is less trivial to generalize. It is no longer true that for a given dimension $n$ there 
are only finitely many $n$-dimensional almost--crystallographic groups 
(or even almost--Bieberbach groups) up to isomorphism. This is not even true when one restricts to
 considering only almost--Bieberbach groups modeled on a given nilpotent Lie group $G$.
Indeed, if one considers the real Heisenberg group $H^\R$, which is the easiest non-abelian nilpotent Lie group, then this group contains all of the groups $H_n$ as a lattice (see page~\pageref{heisn}).  These groups $H_n$ are 
pairwise non-isomorphic almost--Bieberbach groups modeled on $H^\R$ and hence the corresponding 
nilmanifolds $H_n\backslash \R^n$ are pairwise non-homeomorphic. 

\medskip

So the big difference with the abelian case is that a fixed nilpotent Lie group $G$ can have infinitely many non-isomorphic lattices, while all lattices of $\R^n$ are isomorphic to the same group $\Z^n$. Algebraically, we can also formulate the third Bieberbach theorem as follows. Fix a positive integer $n$, then there are only finitely 
many crystallographic groups containing $\Z^n$ as its group of pure translations.  
From this point of view, the third Bieberbach theorem does have a direct generalization.
\begin{Thm}
Let $N$ be a finitely generated, torsion free nilpotent group. Then there are, up to isomorphism, only finitely 
many almost--crystallographic groups $\Gamma$ for which the group of pure translations of $\Gamma$ is 
isomorphic to $N$. 
\end{Thm}
This theorem was first formulated by K.B.~Lee in \cite{lee88-1} (but see also \cite{dim92-2}) using the concept of an essential extension.
An essential extension of a finitely generated, torsion free nilpotent group $N$ is an extension of the form 
\[ 1 \to N \to E \to F\to 1\]
in which $F$ is a finite group and $N$ is maximal nilpotent in $E$. So any almost 
crystallographic group $\Gamma$ modeled on $G$ gives rise to such an essential extension:
$1 \to N=\Gamma \cap G \to \Gamma \to F = \Gamma/(G\cap N) \to 1$ by the generalized first Bieberbach theorem and Proposition~\ref{maxnilp}. Later on, we will see that any group $E$ which is obtained as such an essential 
extension of $N$ is (isomorphic to) an almost--crystallographic group. 

\bigskip

As already promised, we end this section with some explicit examples of almost--crystallographic and almost--Bieberbach groups.

Let $N=H$ be the Heisenberg group, which is a lattice of its Mal'cev completion $H^\R=UT_3(\R)$. 
Let $\varphi\in \Aut(H^\R)$ be the automorphism of order 3, which we introduced on page~\pageref{auto3} (on that 
page we defined it on $H^\Q$, but the definition also works on $H^\R$ of course). 
Let $F=\{1,\varphi,\varphi^2\}\subseteq \Aut(H^\R)$, then $F\cong \Z_3$. As we already saw, we have that 
$\varphi(N)=N$. It follows that we can form the semidirect product 
\[\Gamma_1=N \semi F\subset H^\R \semi F \subseteq \Aff(H^\R).\]
It is not difficult to see that $\Gamma_1$ is discrete and cocompact in $H^\R \semi F$ and hence
$\Gamma_1$ is an almost-crystallographic group. It is not an almost-Bieberbach group since $\Gamma_1$ has torsion 
(e.g.\ $\varphi$ is an element of order 3 in $\Gamma_1$).

\medskip

Take
\[ c^{\frac13}=\left( \begin{array}{ccc} 1 & 0 & \frac13 \\ 0 & 1 & 0 \\
0 & 0 & 1 \end{array} \right)\in H^{\R}\]
(This notation comes from the fact that $(c^{\frac13})^3=c$, where $c$ (and $a,b$) were defined on page~\pageref{abc}).

Let $\alpha=c^{\frac13}\varphi\in H^\R\semi F$. Let 
\[ \Gamma_2 =\langle N , \alpha\rangle \]
One computes that $\alpha^3=c$ and that $\alpha N \alpha^{-1} =N$. Using this, it is not so difficult to see 
that $N$ is a normal subgroup of $\Gamma_2$ of index 3 and that 
$\Gamma_2$ is indeed an almost-crystallographic group. Note that for any element $\gamma$ of $\Gamma_2$ there exists 
a $\lambda\in N$ such that one of the three following expressions holds:
\[\gamma=\lambda,\;\; \gamma=\lambda\alpha\mbox{ \ or \ }\gamma=\lambda\alpha^2.\]

Moreover, we claim that this group $\Gamma_2$ is 
torsion free and hence
$\Gamma_2$ is an almost-Bieberbach group and  $\Gamma_2\backslash H^\R$ is an infra-nilmanifold with 
holonomy group $F\cong\Z_3$. How do we check that $\Gamma_2$ is torsion free? It turns out that this 
can be done quite easily by using the $4$-dimensional matrix representation $\psi$ of $H^\R$ which we defined on 
page~\pageref{reppsi} (where one has to replace $\Q$ by $\R$).
We can extend this representation $\psi:H^\R \to \GL(\R^4)$ to a representation (which we also denote by $\psi$)
\[ \psi:H^\R\semi F \to \GL(\R^4)\]
by mapping $\alpha$ to $A$ (see page~\pageref{defA}). The fact that this is really a representation follows from the 
fact that $\psi(\varphi(h))= A \psi(h) A^{-1}$ for all $h\in H^\R$ (which was explained at the end of section~\ref{voorbeeldsec}).
Now let us assume that $\Gamma_2$ is not torsion free, then there has to exist a non-trivial  element 
$\gamma$ of $\Gamma_2$ 
which is of finite order. As $\gamma^3\in N$ and $N$ is torsion free, we must have that 
$\gamma^3=1$. Moreover, $\gamma$ cannot belong to $N$, so we must have that there exists 
a $\lambda\in N$ such that 
$\gamma=\lambda\alpha$ or $\gamma=\lambda\alpha^2$. By replacing $\gamma$ with $\gamma^{-1}$ if necessary, me may assume that $\gamma=\lambda\alpha$.
It follows that 
\[\psi(\gamma)=\psi(\lambda\alpha)= 
\left(\begin{array}{cccc} 
1 & 1 - \frac{x}{2} & \frac{x-y-1}{2} & \frac13 -\frac{xy}{2} +z\\ 0 & 0 & -1 & x \\
0 & 1 & -1 & y \\
0 & 0 & 0 & 1
\end{array}
\right),\mbox{ for some }x,y,z\in\Z.\]
Using this we find that 
\[ \label{gebruikrep}\psi(\gamma^3)=
\left(\begin{array}{cccc} 
1 & 0 & 0 & 1+\frac{ 3(x-y) -(x+y)^2 }{2} +3 z\\
0 & 1 & 0 & 0 \\
0 & 0 & 1 & 0 \\
0 & 0 & 0 & 1 
\end{array}
\right).\]
We leave it to the reader to check that for any $x,y\in \Z$ the expression $\frac{ 3(x-y) -(x+y)^2 }{2} $ is
either equal to $0$ modulo $3$ or equal to $1$ modulo $3$. In any case, this means that 
$\psi(\gamma^3)$ is not the indentity matrix for any choice of $x,y,z\in \Z$, from which we deduce that $\gamma^3\neq 1$ 
and hence $\Gamma_2$ is torsion free.
\section{Characterizations of almost--crystallographic groups}
In this section we want to describe several algebraic characterizations of almost--crystallographic groups.
As any almost--crystallographic group $E$ can be seen as an extension $1\to N \to E \to F\to 1$ where $N$ is 
a nilpotent group, it is useful to recall the basic theory of group extensions with a non-abelian kernel $N$ (see e.g~\cite[Section IV.6]{brow82-1}).

For the moment, let 
\[ 1 \to N \to E \to F \to 1\]
be any extension of groups, with no restriction on the groups involved. 
There is a natural homomorphism $\phi: E\to \Aut(N)$ which is induced by conjugation inside $E$:
we have that $\phi(e):N \to N: n \mapsto \phi(e)(n)= e n e^{-1}$, where we will consider 
$N$ as a subgroup of $E$. As $\phi(N)=\Inn(N)$, this induces 
a homomorphism $\psi: F \to \Out(N)=\Aut(N)/\Inn(N)$. We will call a homomorphism $\psi:F \to \Out(N)$ an abstract kernel and we
say that the extension $1 \to N \to E \to F \to 1$ is compatible with the abstract kernel $\psi$ (which we obtained from $\phi$).

\medskip

Now given two groups $F$ and $N$ and an abstract kernel $\psi:F \to \Out(N)$, we want to classify all
group extensions $1 \to N \to E \to F \to 1$ which are compatible with $\psi$. Let us recall that 
two extensions  $1 \to N \to E \to F \to 1$ and  $1 \to N \to E' \to F \to 1$ are said to be equivalent if and 
only if there exists a homomorphism $\alpha: E \to E'$ (which will be necessarily an isomorphism), such that the 
following diagram commutes:
\[
\xymatrix{ 1\ar[r] &  N \ar[r] \ar@{=}[d] & E \ar[r]\ar[d]_\alpha & F \ar[r]\ar@{=}[d] & 1\\
1\ar[r] &  N \ar[r] & E' \ar[r] & F \ar[r] & 1.}
\]
Let us denote by $\Ext_\psi(F,N)$ the set of equivalence classes of extensions of $N$ by $F$ compatible with $\psi$. It might happen that this set is empty. But suppose for the moment that it is not empty and fix such an extension
$1 \to N \to E \to F \to 1$. Choose a function $s:F \to E$, such that $s(1)=1$ and $s(f)$ is mapped to $f$ under the 
projection $E\to F$. This $s$ determines a function $\varphi:F \to \Aut(N): f \mapsto \varphi(f)$, where 
$\varphi (f):N\to N: n \mapsto s(f) n s(f)^{-1} $.  If $p:\Aut(N) \to \Out(N)$ denotes the natural projection, we then have that $p\circ \varphi = \psi$. The function $s$ also determines a second function, {\em a non-abelian 2-cocycle}, 
$c: F\times F \to N$, which is given by $c(f,g) s(fg) = s(f) s(g)$ for all $f,g\in F$.  The function $\varphi$ and 
$c$ satisfy the following two conditions:
\begin{enumerate}
\item $\forall f,g\in F$: $\varphi(f)\varphi(g) = \mu(c(f,g)) \varphi(fg)$ where $\mu(x):N\to N :n \mapsto x n x^{-1} $ 
denotes the inner autormorphism determined by $x$.
\item $\forall f,g,h\in F$: $c(f,g) c(fg,h) = (\varphi(f) ( c(g,h) )) c(f,gh)$.
\end{enumerate}
Conversely, given two functions $c:F\times F \to N$ and $\varphi:F \to \Aut(N)$ with $p\circ \varphi = \psi$, satisfying the two conditions above, one 
constructs the extension $E_{c,\varphi}$, where the underlying set is $N\times F$ and the 
product is given by 
\[ \forall m,n\in N,\;\forall f,g\in F:\; (n,f)(m,g)= ( n\, (\varphi(f)(m)) \,c(f,g), fg).\]
E.g., when there is a homomorphism $\varphi:F \to \Aut(N)$, lifting $\psi$, then one can take $c(f,g)=1$ for all $f,g\in F$ and 
then $E_{c,\varphi}$ is just $N\semi_\varphi F$.

\medskip

The main result on the classification of extensions with a non-abelian kernel says that if $\Ext_\psi(F,N)$ is non-empty, then this set is in 1-1 correspendence with $H^2(F,Z(N))$.

\medskip

A proof of the following result can be found in \cite[Lemma 3.1.2]{deki96-1} for the case $\F=\R$, but the proof also works for $\F=\Q$.
\begin{Lem}\label{split} Let $N$ be a finitely generated torsion free nilpotent group and let $F$ be a finite group. Let $\F=\Q$ or $\R$ and 
suppose that $\psi: F \to \Out(N^\F)$ is an abstract kernel. Then 
\begin{enumerate}
\item $\psi$ lifts to a homomorphism $\varphi:F \to \Aut(N^\F)$ and
\item any extension of $N^\F$ by $F$ is a split extension (i.e.\ is equivalent to a semi-direct product $N^\F\semi_\varphi F$).
\end{enumerate}
\end{Lem}

\medskip

We are now going to prove our first algebraic characterization of almost--crystallographic groups. The proof of this theorem is 
of a constructive nature and has thus importance from this point of view also.

\begin{Thm}\label{algchar}
Let $E$ be a finitely generated virtually nilpotent group. Then, the following are equivalent:
\begin{enumerate}
\item $E$  is (isomorphic to) an almost--crystallographic group. 
\item $E$ contains a finitely generated torsion free nilpotent normal  subgroup $N$ such that $N$ is maximal nilpotent in $E$ and $[E:N]<\infty$. (In other words if and only if $E$ can be seen as an essential extension).
\item $E$ does not contain a nontrivial finite normal subgroup.
\end{enumerate}
\end{Thm}
\begin{proof}
(1. $\Rightarrow$ 2.) When $E\cong \Gamma$, then the generalized first Bieberbach Theorem \ref{bieb1} and Proposition \ref{maxnilp} show the 
existence of the subgroup $N$ (which corresponds to the group of pure translation $\Gamma\cap G$ in $\Gamma$).

\medskip

(2. $\Rightarrow$ 3.) Assume that $E$ does contain a non-trivial finite normal subgroup $G$, then 
$[G,N]\subseteq N \cap G =1 $ which implies that $\forall g\in G$ the group $\langle N, g \rangle$ is nilpotent, contradicting 
the maximal nilpotency of $N$.

\medskip

 (3. $\Rightarrow$ 1.) Now, assume that $N$ is a nilpotent normal subgroup of finite index in $E$. Then $N$ is torsion free, because 
 the set of torsion elements of $N$ is a finite normal subgroup of $E$. Then we consider the  extension
\[ 1 \to N \to E \to F \to 1\]
and choose a section $s:F \to E$ giving rise to a map $\varphi: F \to \Aut(N)$ (which is a lift of the abstract kernel of the 
extension) and a non-abelian 2-cocycle $c:F\times F\to N$ satisfying the two conditions mentioned above.
If we denote the extension of $F$ by $N$ determined by $c$ and $\varphi$ as $E_{c,\varphi}$
 as before, we know that there is a
 commutative diagram of the form (where $\alpha$ is an isomorphism):
 \[ \xymatrix{  1\ar[r] &  N \ar[r] \ar@{=}[d] & E \ar[r]\ar[d]_\alpha & F \ar[r]\ar@{=}[d] & 1\\
1\ar[r] &  N \ar[r] & E_{c,\varphi} \ar[r] & F \ar[r] & 1.}\]
Now, since any automorphism $\varphi(f)\in \Aut(N)$ can be extended uniquely to an automorphism 
$\tilde\varphi(f) \in \Aut(N^\R)$, we obtain a map $\tilde{\varphi}:F\to \Aut(N^\R)$, such that 
the composition $\psi:F \to  \Aut(N^\R) \to \Out(N^\R)$ is a homomorphism (so an abstract kernel).

The map 
$c: F \times F \to N$ can also be interpreted as a map $\tilde{c}: F \times F \to N^\R$ (using the embedding of $N$ into $N^\R$). Of course, the pair $\tilde{c}, \tilde{\varphi}$ still satisfies the two 2-cocycle conditions and we can define the 
group $\tilde{E}_{\tilde{c},\tilde{\varphi}}$ which is an extension of $N^\R$ by $F$. Now, let 
$i:N\to N^\R$ denote the embedding, then we have a commutative diagram 
 \[ \xymatrix{  1\ar[r] &  N\ar[r] \ar@{ >->}_i [d]  & E_{c,\varphi} \ar[r]\ar@{ >->}[d]^{i\times 1_F} & F \ar[r]\ar@{=}[d] & 1\\
1\ar[r] &  N^\R \ar[r] & \tilde{E}_{\tilde{c},\tilde{\varphi}} \ar[r] & F \ar[r] & 1.}\]
By Lemma~\ref{split}, we know that the bottom sequence splits and so there is a commutative diagram 
in which $\beta$ is an isomorphism and $\tilde{\varphi}':F\to \Aut(N^\R)$ is another lift of the abstract kernel 
$\psi:F\to \Out(N^\R)$, which is a homomorphism of groups:
\[ \xymatrix{  1\ar[r] &  N^\R\ar[r] \ar@{ =}[d]  & \tilde{E}_{\tilde{c},\tilde{\varphi}}  \ar[d]^{\beta}\ar[r] &
 F \ar[r]\ar@{=}[d] & 1\\
1\ar[r] &  N^\R \ar[r] & N^\R\semi_{\tilde{\varphi}'} F \ar[r] & F \ar[r] & 1.}\]
Let $\gamma=1 \times \tilde{\varphi}': N^\R \semi_{\tilde{\varphi}'} F \to N^\R\semi \Aut(N^\R): (n,f) \mapsto 
(n,\tilde{\varphi}'(f))$. This is a homomorphism, inducing a commutative diagram 
\[ \xymatrix{  1\ar[r] &  N^\R\ar[r] \ar@{ =}[d]  & N^\R\semi_{\tilde{\varphi}'} F  \ar[d]^{\gamma}\ar[r] &
 F \ar[r] \ar[d]_{\tilde{\varphi}'} & 1\\
1\ar[r] &  N^\R \ar[r] & N^\R\semi \Aut(N^\R) \ar[r] & \Aut(N^\R) \ar[r] & 1.}\]
Now, let $j:E\to N^\R \semi \Aut(N^\R)$ be the composition
$j= \gamma \circ \beta \circ (i\times 1_F) \circ \alpha$:
\[\xymatrix{
E\ar[r]^\alpha 
\ar@{<}@<0pt> `u[r] `[rrrr]^j  [rrrr]
 & E_{c,\alpha}\ar[r]^{i\times 1_F} &\tilde{E}_{\tilde{c},\tilde{\varphi}} \ar[r]^\beta & 
N^\R\semi_{\tilde{\varphi}'} F \ar[r]^\gamma & N^\R \semi \Aut(N^\R)  \\
N \ar[r]_{1_N}\ar@{ >->}[u]   & N \ar@{ >->}[r]_{i} \ar@{ >->}[u]  & N^\R \ar[r]_{1_{N^\R}}\ar@{ >->}[u]   &
 N^\R \ar[r]_{1_{N^\R}}\ar@{ >->}[u] & N^\R \ar@{ >->}[u]
}\]
Note that when we restrict $j$ to $N$, we obtain the embedding 
$j_{|N}=i:N\to N^{\R}$ of $N$ into its real Mal'cev completion. This shows $\Ker(j) \cap N=1$, from which it follows that 
$\Ker(j)$ is finite. However, as we are assuming that $E$ does not have any non-trivial finite subgroups, this implies that 
$j$ is injective and hence $E\cong j(E)$. Note that $j(E)\subseteq N^\R \semi \tilde{\varphi}'(F)$. As $\tilde{\varphi}'(F)$
is finite, we can choose a maximal compact subgroup $C\subseteq \Aut(\R^n)$ containing  $\tilde{\varphi}'(F)$ and
then we have that $j(E)\subseteq N^\R \semi C$. Now, $j(E)\cap N^\R$ is a group containing $N$ as a subgroup 
of finite index. As $N$ is a lattice in $N^\R$, also $j(E)\cap N^\R$ is a lattice of $N^\R$. And hence 
$j(E)\cap N^\R$ is a discrete and cocompact subgroup of $N^\R \semi C$. However, as $j(E)$ contains
$j(E)\cap N^\R$ as a subgroup of finite index, also $j(E)$ itself is  discrete and cocompact in $N^\R \semi C$, showing
that $E\cong j(E)$ is an almost--crystallographic group modeled on $N^\R$.

\end{proof}

The above characterization of almost-crystallographic groups immediately implies the following description of 
almost-Bieberbach groups.

\begin{Cor}
A group $E$ is isomorphic to an almost--Bieberbach group if and only if $E$ is a finitely generated torsion free virtually nilpotent group.
\end{Cor}

Also the following result is easy to deduce now:

\begin{Cor}
Let $\Gamma$ be an almost--crystallographic group and let $E$ be any normal subgroup of $\Gamma$, then $E$ is also 
an almost--crystallographic group.
\end{Cor}
\begin{proof}
As $\Gamma$ is finitely generated virtually nilpotent, the group $E$ is also finitely generated and virtually nilpotent.
Let $G$ be the unique maximal finite normal subgroup of $E$ (which exists in any polycyclic-by-finite group), then $G$ is 
characteristic in $E$ and hence normal in $\Gamma$. But since $\Gamma$ is almost--crystallographic, this implies that 
$G=1$, which in turn implies that $E$ is almost--crystallographic.
\end{proof}

\medskip

The next result we want to mention is useful for constructing almost--crystallographic groups by induction 
on the Hirsch length (which equals the dimension).

To explain this result, we need some more background on nilpotent groups. 
Let $N$ be a finitely generated torsion free nilpotent group, then generally $N/\gamma_i(N)$ is not torsion free.
But we can introduce the groups 
\[ \sqrt{\gamma_i(N)} = \{ n \in N\;|\; \exists k>0:\; n^k \in \gamma_i(N)\}.\]
Then $\sqrt{\gamma_i(N)} = p^{-1} \tau (N/\gamma_i(N))$, where $\tau (N/\gamma_i(N))$ is the finite group 
of all torsion elements of $N/\gamma_i(N)$ and $p:N \to N/\gamma_i(N)$ is the natural projection. It follows 
that $N/\sqrt{\gamma_i(N)}$ is torsion-free and $\gamma_i(N)$ is contained in $\sqrt{\gamma_i(N)}$ as a subgroup 
of finite index. In fact, $\sqrt{\gamma_i(N)}$ is the smallest normal subgroup $H$ of $N$, containing $\gamma_i(N)$ as a 
subgroup of finite index and for which $N/H$ is torsion free.

This group can also be described using the rational or real Mal'cev completion of $N$:
\begin{Lem} Let $N$ be a finitely generated torsion free nilpotent group, then for all $i>0$:
\[ \sqrt{\gamma_i(N)} = N \cap \gamma_i(N^\Q) = N \cap \gamma_i(N^\R) .\]
\end{Lem}
Using this lemma it follows immediately that  $[ \sqrt{\gamma_i(N)}, \sqrt{\gamma_j(N)}]\subseteq 
\sqrt{\gamma_{i+j}(N)}$. 

\begin{Thm} Let $\Gamma$ be an almost--crystallographic group with translation subgroup $N$. Then for 
any $i>1$, the group $\Gamma/\sqrt{\gamma_i(N)}$ is also an almost--crystallographic group.
\end{Thm}
\begin{proof}
There is a short exact sequence $1 \to N \to \Gamma \to F\to 1$ where $F$ is a finite group and $N$ is maximal 
nilpotent in $\Gamma$. This gives rise to a short exact sequence
$1 \to N /\sqrt{\gamma_i(N)}\to \Gamma/\sqrt{\gamma_i(N)} \to F \to 1$, so that $\Gamma/\sqrt{\gamma_i(N)}$
contains the finitely generated torsion free nilpotent group $N /\sqrt{\gamma_i(N)}$ as a subgroup of finite index.
By Theorem~\ref{algchar} it suffices to show that $\Gamma/\sqrt{\gamma_i(N)}$ does not contain 
a nontrivial finite normal subgroup $G$. Suppose on the contrary that such a $G$ does exist, then 
$[ N /\sqrt{\gamma_i(N)}, G]=1$. So any element $g\in G$ commutes with $N /\sqrt{\gamma_i(N)}$.
Now, let $e\in E$, such that its image in $\Gamma/\sqrt{\gamma_i(N)} $ is a nontrivial element in 
$G$. Then $e\not\in N$ and conjugation with $e$ in $E$ induces the identity automorphism on 
$N/\sqrt{\gamma_i(N)}$. It follows that it also induces the identity automorphism on $N/\sqrt{\gamma_2(N)}$.
From, this it now follows that conjugation with $e$ induces the identity automorphism on each of the 
quotients $\sqrt{\gamma_j(N)}/\sqrt{\gamma_{j+1}(N)}$. This is enough to conclude that the group 
$\langle N, e\rangle$ is nilpotent, contradicting the maximal nilpotency of $N$ in $E$. 
\end{proof}

\medskip

This theorem shows that if $\Gamma$ is an almost--crystallographic group, with a $c$--step nilpotent 
translation subgroup $N$, $\Gamma$ fits in a short exact sequence 
\[ 1 \to \sqrt{\gamma_c(N)} \to \Gamma \to \Gamma/\sqrt{\gamma_c(N)} \to 1. \]
Here $\sqrt{\gamma_c(N)}\cong \Z^k$ (for some $k$) and $\Gamma/\sqrt{\gamma_c(N)}$ is an almost--crystallographic subgroup containing $N/\sqrt{\gamma_c(N)}$ as its translation subgroup. The short exact sequence
is almost central, since $N/\sqrt{\gamma_c(N)}$ acts trivially on $\sqrt{\gamma_c(N)}$. This implies that to construct all possible $\Gamma$, we need to construct all possible almost-central extensions of $\Z^k$ by all possible 
$\Gamma/\sqrt{\gamma_c(N)}$, which leads to a classification procedure by induction on the dimension. This approach was used to obtain the classification in dimensions 3 and 4 in \cite{deki96-1}.

\medskip

To end this section, we introduce the so called rational representation of an almost--crys\-tal\-lo\-gra\-phic group.
In the proof of Theorem~\ref{algchar}, we constructed for any abstract crystallographic group 
an embedding $j: E \to N^\R \semi \Aut(N^\R)$. One can also do this proof by using $N^\Q$ in stead of $\N^\R$.
In this way we will obtain an embedding $i:E \to N^\Q \semi \Aut(N^\Q)$, in such a way that the 
composite map $E\to N^\Q \semi \Aut(N^\Q) \to \Aut(N^\Q)$ has a finite image, say $F$. 
If we start this construction with $N$ being the unique normal and maximal nilpotent subgroup of $E$, then 
$F=E/N$. If $N'$ is another normal nilpotent group of finite index in $E$, then $N'\subseteq N$, $N'^\Q=N^\Q$ 
and $E/N'$ is a group which maps onto $F$ (indeed $(E/N')/(N/N')=F$).

\begin{Def} Let $\Gamma$ be an almost--crystallographic group. A rational realization of $\Gamma$ is an embedding 
\[ i: \Gamma \to N^\Q\semi \Aut(N^\Q) \] 
where $N$ is a finitely generated torsion free nilpotent group and the composition 
$\Gamma \to N^\Q\semi \Aut(N^\Q) \to \Aut(N^\Q)$ has finite image, say $F$.
\end{Def}
We just explained that any almost-crystallographic group has a rational representation. Moreover,
if we consider the composition $j: \Gamma \stackrel{i}{\to}N^\Q\semi \Aut(N^\Q) \to
N^\R\semi \Aut(N^\R) $, then $j(\Gamma)$ is a genuine almost--crystallographic subgroup with translation subgroup 
$j(\Gamma)\cap N^\R = i(\Gamma) \cap N^\Q$. It follows that the group $F$ is the holonomy group of $\Gamma$ 
(see the generalized first Bieberbach theorem). The induced representation 
\[ \phi: F \to \Aut(N^\Q) \] 
is called the rational holonomy representation\index{rational holonomy representation}. This representation depends on the choice of $\varphi$ and 
its chosen lift $\tilde{\varphi}'$. Other choices will alter $\varphi$ by an inner automorphism of $N^\Q$.
Note that since $\Aut(N^\Q)\cong \Aut(\lien^\Q)$, we can view the rational holonomy representation as a 
representation into a rational vector space. Analogously, one can define the real holonomy representation\index{real holonomy representation}  (which 
is often just called the holonomy representation\index{holonomy representation}). 

\medskip

It is often useful to consider only the induced representation 
\[ \varphi_{ab}:F \to \Aut(N^\Q /[N^\Q,N^\Q]) = \Aut(\lien^\Q/[\lien^\Q,\lien^\Q]) \]
which will be referred to as the abelianized rational holonomy representation.\label{abrathol} This abelianized
representation is independent of any choices made for $\varphi$ or $\tilde{\varphi}'$.

\section{Maps on infra--nilmanifolds}\label{maps}
In this section we will describe all possible maps between two infra--nilmanifolds up to homotopy.
In order to do this, we need to generalize the second Bieberbach theorem even further.

Let $G_1$ and $G_2$ be two simply connected nilpotent Lie group, then we will use 
$\End(G_1,G_2)$ to denote the set of continuous homomorphisms from $G_1$ to $G_2$.
We now let $\aff(G_1,G_2)=G_2\times \End(G_1,G_2)$ and any $(d,\delta)\in \aff(G_1,G_2)$ 
(so $d\in G_2$ and $\delta\in \End(G_1,G_2)$) determines a so called affine map
\[ (d,\delta):G_1 \to G_2: g \mapsto d\cdot \delta (g).\]
Note that in case $G_1=\R^k$ and $G_2=\R^l$, $\aff(\R^k,\R^l)$ is the usual set of affine maps from $\R^k$ to $\R^l$.

When $G_1=G_2=G$ we will use $\aff(G)$ to denote $\aff(G,G)$. In this case $\aff(G)=G\semi \End(G)$ (where $\End(G)=\End(G,G)$) is a semigroup containing the group $\Aff(G)$. Then, $\Aff(G)$ consists exactly of the invertible elements of $\aff(G)$. It was K.B.~Lee who first proved the more general version 
of the generalized second Bieberbach theorem in \cite{lee95-2} in the case $G_1=G_2$. A slightly adapted proof for the case 
$G_1\neq G_2$ can be found e.g.\ in \cite{ll09-1}. (In fact, in this last paper the groups $G_1$ and $G_2$ are assumed to 
be solvable of type (R), which is even more general than being nilpotent). As I cannot improve their arguments, I refer 
to those papers for the proof.
\begin{Thm}\label{keyLee}
Take $i=1$ or $2$ and Let $G_i$ be a simply connected nilpotent Lie group, $C_i\subseteq \Aut(G_i)$ be a maximal 
compact subgroup and $\Gamma_i\subseteq G_i\semi C_i$ be an almost--crystallographic group.
Then, for any homomorphism $\theta:\Gamma_1\to \Gamma_2$, there exists 
an affine map $(d,\delta) \in \aff(G_1,G_2)$ such that 
\[ \forall \gamma\in \Gamma_1: \; \theta(\gamma) \circ (d,\delta) = (d, \delta) \circ \gamma.\]
\end{Thm}

Let $\Gamma_i$ ($i=1,2$) be as in the theorem, then there are short exact sequences 
$1 \to N_i \to \Gamma_i \to F_i\to 1$, where $N_i$ is the translation subgroup of $\Gamma_i$ and 
$F_i$ is the finite holonomy group of $\Gamma_i$. Let $k=|F_1|\cdot |F_2|$. 
let $N$ be the subgroup of $\Gamma_1$ which is generated by the 
set $\{\gamma^k\;|\; \gamma \in \Gamma_1\}$. Then $N\subseteq N_1$ is a subgroup of finite index in 
$\Gamma_1$ (and hence also in $N_1$) and $\theta(N_1)\subseteq N_2$. Let $\alpha=\theta_{|N}$, then 
as $N^\R=N_1^\R=G_1$, there is a unique extension of  $\alpha$ to a homomorphism $\tilde{\alpha}:G_1 \to G_2$.
For all $n\in N$ we have that 
\[
(\tilde{\alpha}(n) d, \delta) = (\tilde{\alpha} (n),1)  \circ (d,\delta)  =  \theta (n,1) \circ (d,\delta) \\
 =  (d, \delta ) \circ (n,1) = (d \delta(n) ,\delta) 
\]
showing that $\tilde{\alpha}(n)=d \delta(n) d^{-1}$ for all $n\in N$. Since $N$ is 
a lattice in $G_1$, it follows that $\tilde{\alpha}(g)=d \delta(g) d^{-1}$ for all $g\in G_1$. Hence,
$\delta$ is determined, up to an inner automorphism of $G_2$, by $\alpha$, the restriction of 
$\theta$ to $N$. We have the following:
\begin{Lem}
With the notations of Theorem~\ref{keyLee}:
\begin{enumerate}
\item $\delta$ is injective $\Leftrightarrow$ $\theta$ is injective $\Leftrightarrow$ $\theta_{| H}$ is injective, where 
$H$ is any finite index subgroup of $\Gamma_1$.
\item $\delta$ is surjective if and only if $\theta(\Gamma_1)$ is of finite index in $\Gamma_2$.
\end{enumerate}
\end{Lem}
\begin{proof} Let $H$ be a finite index subgroup of $\Gamma_1$
We first of all show that $\theta$ is injective $\Leftrightarrow$ $\theta_{| H}$ is injective. Of course, one direction is 
obvious. Now suppose that $\theta_{| H}$ is injective. Since $H$ is of finite index in $\Gamma_1$, this implies that 
the kernel of $\theta$ is finite. But $\Gamma_1$ does not contain any  non trivial finite normal subgroup,
so $\theta$ is injective. Now, let $N$ be the subgroup of $\Gamma_1$ with $\theta(N)\subseteq N_2$ as mentioned 
above and let $\tilde{\alpha}:G_1 \to G_2$  be the unique lift of $\alpha=\theta_{|N}$. Then 
$\tilde{\alpha}$ is injective if and only if $\alpha$ is injective. Moreover, as explained above
$\tilde{\alpha}(g)=d \delta(g) d^{-1}$ for all $g\in G_1$ from which it follows that $\tilde{\alpha}$ is injective if and 
only if $\delta$ is injective. We can conclude that $\delta$ is injective if and only if $\alpha$ is injective 
if and only if $\theta$ is injective.

\medskip

Analogously, $\delta$ is surjective if and only if $\tilde{\alpha}$ is. Now
$\tilde{\alpha}$ is surjective if and only if $\alpha(N)$ is a lattice of $G_2$. As 
$\alpha(N)=\theta(N)\subseteq N_2$, we have that $\theta(N)$ is a lattice of $G_2$ if and only if $\theta(N)$ 
is of finite in $N_2$, which is equivalent to being of finite index in $\Gamma_2$. But $\theta(N)$ being 
of finite index in $\Gamma_2$ is equivalent to $\theta(\Gamma_1)$ being of finite index in $\Gamma_2$, so we proved that 
$\delta$ is surjective if and only if $[\Gamma_2:\theta(\Gamma_1)]<\infty$.
\end{proof}

\begin{Rmk}
When $G_1=G_2$ and $\theta:\Gamma_1\to \Gamma_2$ is an isomorphism, then the above lemma 
implies that $\delta\in \Aut(G)$ and so $(d,\delta)\in \Aff(G)$ and we have that 
\[ \forall \gamma \in \Gamma_1: \; \theta(\gamma) \circ (d,\delta) = (d,\delta) \circ \gamma \mbox{ or 
equivalently } \theta(\gamma) = (d,\delta ) \gamma (d,\delta)^{-1},\]
which is exactly what Theorem~\ref{tweedebieb} says.
\end{Rmk}

Now, we will use this in the study of maps between two infra-nilmanifolds and therefore from now onwards
we assume that $\Gamma_1$ and $\Gamma_2$ are almost-Bieberbach groups modeled on $G_1$ and $G_2$ respectively.
Let $p_i:G_i \to \Gamma_i\backslash G_i$ ($i=1,2$) denote the natural projections ($G_i$ is the 
universal covering space of $\Gamma_i\backslash G_i$).

\medskip

The generalized second Bieberbach Theorem gives a  way of constructing maps between two infra--nilmanifolds
\begin{Lem} \label{affmap}
For $i=1,2$, let $\Gamma_i$ be an almost--Bieberbach group modeled on a simply connected nilpotent Lie group $G_i$ and
let $\theta:\Gamma_1\to \Gamma_2$ be any homomorphism. Consider  $(d,\delta)\in \aff(G_1,G_2)$ with 
\[ \forall \gamma\in \Gamma_1: \; \theta(\gamma) \circ (d,\delta) = (d, \delta) \circ \gamma,\]
then $(d,\delta)$ induces a map 
\[ \overline{(d,\delta)} :\; \Gamma_1\backslash G_1 \to \Gamma_2\backslash G_2: \Gamma_1\cdot g \mapsto \Gamma_2 \cdot ((d,\delta)\cdot g)=
\Gamma_2 \cdot d\delta(g).\]
\end{Lem}
\begin{proof}
We have to show that  $\overline{(d,\delta)}$ is well-defined. So take $g\in G_1$ and $\gamma\in \Gamma_1$, then 
we need to check that $(d,\delta)\cdot g$ and $(d,\delta) \cdot ( \gamma\cdot g)$ lie in the same $\Gamma_2$-orbit.
But this is obvious since:
\[ (d,\delta) \cdot ( \gamma\cdot g) = ((d,\delta)\circ \gamma) \cdot g = (\theta(\gamma) \circ (d,\delta))\cdot g =
\theta(\gamma) \cdot ( (d,\delta)\cdot g) \in \Gamma_2 \cdot ((d,\delta)\cdot g).\]
\end{proof}

\begin{Def}
We will refer to a map $\overline{(d,\delta)}:\Gamma_1\backslash G_1 \to \Gamma_2\backslash G_2$ as in the lemma above, as an
affine map between the two infra-nilmanifolds\index{affine map between infra-nilmanifolds}.
\end{Def}

\medskip

Consider now any (continuous) map
\[ f: \Gamma_1\backslash G_1 \to \Gamma_2\backslash G_2\]
and fix a lift $\tilde{f}_0:G_1 \to G_2$  of $f$. So we have a commutative diagram 
\[
\xymatrix{ G_1 \ar[r]^{\tilde{ f}_0} \ar[d]_{p_1} & G_2\ar[d]^{p_2}\\
\Gamma_1\backslash G_1 \ar[r]_f & \Gamma_2\backslash G_2}
\]
Such a lift $\tilde{f}_0$ always exists and if we compose $\tilde{f}_0$ with $\beta\in \Gamma_2$, then $\tilde{f}=
\beta \circ \tilde{f}_0$ will be another lift of $f$. And in fact all lifts $\tilde{f}$ of $f$ can be uniquely written 
as a composition $\beta \circ \tilde{f}_0$. So there is a 1--1 correspondence between the lifts of $f$ and $\Gamma_2$.
On the other hand, for any $\alpha \in \Gamma_1$, we also have that $\tilde{f}_0 \circ \alpha$ is a lift of 
$f$. Hence, there exists a $\beta \in \Gamma_2$ such that $\tilde{f}_0 \circ \alpha = \beta \circ \tilde{f}_0$.

It follows that the map $f$ induces a homomorphism $f_\ast: \Gamma_1 \to \Gamma_2$ which is 
determined by 
\[ \forall \gamma\in \Gamma_1: \; \ \tilde{f}_0 \circ \gamma = f_\ast(\gamma) \circ \tilde{f}_0 .\]
The homomorphism $f_\ast$ depends on the choice of the reference lift $\tilde{f}_0$. Another choice of lift will change 
$f_\ast$ by an inner automorphism of $\Gamma_2$.
\begin{Def}
We will refer to $f_\ast:\Gamma_1\to \Gamma_2$ as the induced homomorphism by $f$ (with respect to the reference 
lift $\tilde{f}_0$).
\end{Def}

\begin{Rmk}
Let $\overline{(d,\delta)}$ be an affine map (associated to $\theta:\Gamma_1 \to \Gamma_2$) as in
Lemma~\ref{affmap}. If one takes $(d,\delta)$ as the reference lift of $\overline{(d,\delta)}$, then the 
equation $ \theta(\gamma) \circ (d,\delta) = (d, \delta) \circ \gamma$ shows that 
\[ \overline{(d,\delta)}_\ast = \theta.\]
It follows that for any homomorphism $\theta: \Gamma_1 \to \Gamma_2$, we can construct 
a map $f:\Gamma_1\backslash G_1 \to \Gamma_2\backslash G_2$ such that $f_\ast = \theta$ (where we can take $f$ to be an 
affine map).
\end{Rmk}

\begin{Rmk}
Under the right identifications of $\Gamma_i$ with the fundamental group of $\Gamma_i\backslash G_i$ ($i=1,2$) we 
have that $f_\ast$ exactly corresponds to the induced map on the 
fundamental group $f_\ast:\Pi_1(\Gamma_1\backslash G_1,x) \to \Pi_1(\Gamma_2\backslash G_2,f(x))$.
\end{Rmk}

\begin{Thm}
For $i=1,2$, let $\Gamma_i$ be an almost--Bieberbach group modeled on a simply connected nilpotent Lie group $G_i$ and
let $f,g:\Gamma_1\backslash G_1 \to \Gamma_2\backslash G_2$ be two maps such that $f_\ast = g_\ast:\Gamma_1 \to \Gamma_2$
(where $f_\ast$ and $g_\ast$  are the induced homomorphisms w.r.t.\ some reference lifts $\tilde{f}$ and $\tilde{g}$).\\
Then $f$ and $g$ are homotopic.
\end{Thm}
\begin{proof}

In fact this theorem is a special case of a more general result for $K(\Pi,1)$--spaces (see \cite{whit78-1}).

\medskip

There is also a constructive proof of this result. Let $\tilde{f}$ and $\tilde{g}$ be the reference 
lifts for $f$ and $g$ respectively, then, since $g_\ast= f_\ast$:
\[ \forall \gamma\in \Gamma:\; f_\ast(\gamma)\circ \tilde{f} = \tilde{f} \circ \gamma\mbox{ and }
  f_\ast(\gamma)\circ \tilde{g} = \tilde{g} \circ \gamma.\]

Consider the homotopy 
\[ \tilde{H}:\; G_1\times I \to G_2 : (x,t) \mapsto \tilde{f}(x) \left( (\tilde{f}(x))^{-1} \tilde{g}(x) \right)^t ,\]
where for all $y\in G_2$, we let $y^t= \exp ( t \log (y))$. Then $\tilde{H}(x,0)=\tilde{f}(x)$ and 
$\tilde{H}(x,1)=\tilde{g}(x)$, so $\tilde{H}$ is a homotopy between $\tilde{f}$ and $\tilde{g}$.
We claim that $\tilde{H}$ induces a homotopy 
\[ H : \Gamma_1\backslash G_1 \to \Gamma_2\backslash G_2: (\Gamma_1\cdot x, t) \mapsto 
\Gamma_2\cdot \tilde{H}(x,t)\]
between $f$ and $g$. It is obvious that $H(\Gamma_1\cdot x,0)=\Gamma_2\cdot \tilde{f}(x) = f(\Gamma_1\cdot x)$ and $H(\Gamma_1\cdot x,1) = g(\Gamma_1 \cdot x)$, so the only thing left to show is that 
$H$ is well defined. To do this, fix a $\gamma_1\in \Gamma_1$ and let $f_\ast(\gamma_1)= (a,\alpha) \in \Gamma_2 \subseteq \Aff(G_2)$. Then $\forall x\in G_1$:
\begin{eqnarray*}
\tilde{H}(\gamma_1\cdot x, t ) & = &  
  \tilde{f}(\gamma_1\cdot x) \left( (\tilde{f}(\gamma_1\cdot x))^{-1} \tilde{g}(\gamma_1 \cdot x) \right)^t\\
&=& \left( f_\ast(\gamma_1) (\tilde{f}(x)) \right)
\left( \left( f_\ast(\gamma_1) (\tilde{f}(x))\right)^{-1} \left(\raisebox{4mm}{$ $}f_\ast(\gamma_1) (\tilde{g}(x)) \right)\right)^t\\
& = & \left( a \alpha( \tilde{f}(x))\right)
\left(  \left(a \alpha( \tilde{f}(x))\right)^{-1} \left(\raisebox{4mm}{$ $} a \alpha( \tilde{g}(x))\right)
\right)^t \\
& = & \left( a \alpha( \tilde{f}(x))\right)
\left(   \alpha( \tilde{f}(x)^{-1} \tilde{g}(x)) \right)^t\\
& = & a \alpha \left( \tilde{f}(x) \left((\tilde{f}(x))^{-1} \tilde{g}(x)  \right)^t \right)\\
& = & f_\ast(\gamma_1) \cdot \tilde{H}(x,t) \in \Gamma_2\cdot \tilde{H}(x,t)
\end{eqnarray*}
from which it follows that $H$ is a well defined homotopy between $f$ and $g$.
\end{proof}

Combining all of the above, we now find that any map between to infra-nilmanifolds is homotpic to an affine map.

\begin{Cor}
For $i=1,2$, let $\Gamma_i$ be an almost--Bieberbach group modeled on a simply connected nilpotent Lie group $G_i$ and
let $f:\Gamma_1\backslash G_1 \to \Gamma_2\backslash G_2$ be any map. Then there exists an affine map $\overline{(d,\delta)}:
\Gamma_1\backslash G_1 \to \Gamma_2\backslash G_2$ which is homotopic to $f$.
\end{Cor}
\begin{proof}
Let $f_\ast:\Gamma_1 \to \Gamma_2$ be the homomorphism induced by $f$ and let $(d,\delta) \in \aff(G_1,G_2)$ be 
the affine map with 
\[ \forall \gamma\in \Gamma_1: \; f_\ast(\gamma) \circ (d,\delta) = (d, \delta) \circ \gamma\]
as in the generalized second Bieberbach Theorem~\ref{keyLee}. Then we know that $(d,\delta)$ determines
an affine map $\overline{(d,\delta)}:\Gamma_1\backslash G_1 \to \Gamma_2\backslash G_2$ with $\overline{(d,\delta)}_\ast = f_\ast$.
It follows that $f$ and $\overline{(d,\delta)}$ are homotopic.
\end{proof}

\begin{Rmk}
In the literature, there has been quite some confusion concerning maps between two infra-nilmanifolds (or selfmaps
of a given infra-nilmanifold). This confusion seems to boil down to the fact that some authors  believed that any 
such map was homotopic to one induced by a homomorphism $\delta: G_1 \to G_2$ (so by an affine map $(1,\delta)$,
with trivial translational part $d=1$), which is incorrect. 
We refer the reader to \cite{deki11-1}  and \cite{deki11-2}
for more information on this.
\end{Rmk}

\section{Computing with almost--crystallographic groups}\label{computation}
When working with infra-nilmanifolds  it is often needed to be able
to make explicit computations involving almost--Bieberbach groups.

\medskip

In this section, we want to explain how this can be done at least in case the Lie group on which the 
almost--Bieberbach group is modeled is of low nilpotency class. 

\medskip

First of all I would like  to mention here the GAP package {\em aclib} \cite{de12-1, gap16-1}. This package contains a 
library of almost--crystallographic groups of dimension at most 4 (including all almost--Bieberbach groups) 
together with some algorithms to compute with them. These almost--crystallographic groups are 
available in two formats. First of all, there is a faithful matrix representation for each of them. Secondly, all of these groups are polycyclic and they are also stored via a polycyclic presentation. As a consequence, all methods of the 
{\em polycyclic} package \cite{en13-1} are also available for these groups. We refer the reader to the manual
of these two packages for more information.

\medskip

The rest of this section will be devoted to the construction of an embedding of the affine group 
$\Aff(G)$ of a 2-step nilpotent Lie group $G$ into the usual affine group $\Aff(\R^n)$, where $n$ is 
the  dimension of $G$. Having also in mind the generalized second Bieberbach theorem, such an embedding is 
very useful to work with explicit examples of infra-nilmanifolds $\Gamma\bs G$ modeled on a 2-step nilpotent 
Lie group $G$ and for constructing, understanding, \ldots\ selfmaps of the manifold $\Gamma\bs G$.
This embedding was first constructed in \cite{deki94-2}, where the proof was quite 
technical. Here we present a new approach to the same embedding.

\medskip

So for the rest of this section we fix a 2-step nilpotent Lie group $G$. 
Let $\lie$ denote the Lie algebra of $G$, then $[\lie,[\lie,\lie]]=0$ and so 
$\lie$ fits in a short exact sequence 
\[ 0 \to [\lie,\lie] \to \lie \to \lie/[\lie,\lie] \to 0 \]
where both $[\lie,\lie]$ and $\lie/[\lie,\lie]$ are abelian Lie algebras. Choose a basis 
$v_1,v_2,\ldots, v_k, w_1, w_2, \ldots, w_l$ for $\lie$, where 
$v_1,v_2, \ldots, v_k$ form a basis of $[\lie,\lie]$, so $k+l=n$ is the dimension of $G$. 
In principle one can choose any basis of $\lie$, but the embedding looks nicer, when using this choice.
With such a choice of basis,  the matrix representation of $\ad_X:\lie \to \lie: Y\to \ad_X(Y)=[X,Y]$ is of the form:
\[ \ad_X=\left(\begin{array}{cc}
0 & M \\ 0 & 0 \end{array}\right)\]
where the 0's denote blocks of zeroes and  $M$ is a $k\times l$ matrix with possibly non-zero entries.

\medskip

Let 
\[ H =\left\{ \left( \begin{array}{cc}
A & a \\ 0 & 1 \end{array}\right) \;|\; A\in \GL(\R^n),\; a\in \R^n \right\} ,\]
then it is easy to see that 
\[\psi: \Aff(\R^n) \to H: (a,A) \mapsto \left( \begin{array}{cc}
A & a \\ 0 & 1 \end{array}\right)\]
is an isomorphism of (Lie) groups.
We will in fact construct an embedding $\varphi: \Aff(G) \to H$.

\medskip

The first step is to define $\varphi$ on $\Aut(G)$. For this we just take the map $\varphi$ which sends an 
automorphism $\alpha$ to its differential $\alpha_\ast$ and we add an extra row and column:
\[ \varphi_1: \Aut(G) \to H: \alpha \mapsto 
\left( \begin{array}{cc} \alpha_\ast & 0 \\ 
0 & 1 \end{array} \right)\]
where $\alpha_\ast$ should be interpreted as the matrix representation of $\alpha_\ast$ with respect to the chosen basis.
It is obvious that $\varphi_1$ is an injective homomorphism.

\medskip

For the translational part, we will first define an embedding of the Lie algebra $\lie$ into $NT_{n+1}(\R)$. 
Consider the following map:
\[ \varphi_{2\ast}: \lie \to NT_{n+1}(\R):X \to \left( \begin{array}{cc}
\frac12 \ad_X & X \\0 & 0\end{array} \right)\]
In this definition, $\ad_X$ is actually the matrix representation of $\ad_X$ and $X$ is the column vector of the 
coordinates of $X$ with respect to the chosen basis. A short computation shows that $\varphi_{2\ast}$ is
a Lie algebra homomorphism (which is obviously injective). Indeed, consider any $X,Y\in \lie$, then  
\begin{eqnarray*}
[\varphi_{2\ast}(X) , \varphi_{2\ast}(Y)] & = & \left[ 
 \left( \begin{array}{cc}
\frac12 \ad_X & X \\0 & 0\end{array} \right)\,,\,
 \left( \begin{array}{cc}
\frac12 \ad_Y & Y \\0 & 0\end{array} \right)\right] \\
 & = & 
 \left( \begin{array}{cc}
\frac14 (\ad_X \ad_Y - \ad_Y\ad_X) & \frac12\ad_X(Y) - \frac12\ad_Y(X) \\0 & 0\end{array} \right)\\
& = & \left( \begin{array}{cc}
\frac14 \ad_{[X,Y]} & \frac12 [X,Y] - \frac12[Y,X] \\0 & 0\end{array} \right)\\
& = & \left( \begin{array}{cc}
\frac12 \ad_{[X,Y]} & [X,Y]  \\0 & 0\end{array} \right)\mbox{ \ \ \ \ \ \ \ \ \ \ (since $\ad_{[X,Y]}=0$)}\\
& = & \varphi_{2\ast} ([X,Y])
\end{eqnarray*}
We let $\varphi_2: G \to UT_{n+1}(\R)\subseteq H$ be the (injective) Lie group homomorphism, whose corresponding differential is $\varphi_{2\ast}$, so 
\[\varphi_2:G \to H: \;g \mapsto \exp( \varphi_{2\ast}( \log (g))).\]
\begin{Thm}\label{explicit}
The map 
\[ \varphi: \Aff(G)=G\semi \Aut(G) \to H:\;(g,\alpha) \mapsto \varphi_2(g) \varphi_1(\alpha)\]
is an injective homomorphism of Lie groups.
\end{Thm}
\begin{proof}
In order to show that $\varphi$ is a homomorphism, we consider any two elements $(x,\alpha),(b,\beta)\in G\semi \Aut(G)$ and have to show that 
\[ \varphi(a \alpha(b), \alpha \beta) = \varphi(a,\alpha)\varphi(b,\beta).\]
Knowing already that $\varphi_1$ and $\varphi_2$ are homomorphisms, this is equivalent to showing that 
\[ \varphi_1(\alpha) \varphi_2(b) \varphi_1(\alpha)^{-1}= \varphi_2 ( \alpha(b) ).\]  
We have (with $B=\log(b)$)
\begin{eqnarray*}
\varphi_1(\alpha) \varphi_2(b) \varphi_1(\alpha)^{-1} & = & 
\left( \begin{array}{cc}
\alpha_\ast & 0 \\ 0 & 1 
\end{array} \right) \exp\left( \begin{array}{cc}
\frac12 \ad_B & B \\0 & 0 
\end{array}\right)\left( \begin{array}{cc}
\alpha_\ast & 0 \\ 0 & 1 
\end{array} \right)^{-1}\\
& = & \exp\left( 
\left( \begin{array}{cc}
\alpha_\ast & 0 \\ 0 & 1 
\end{array} \right) \left( \begin{array}{cc}
\frac12 \ad_B & B \\0 & 0 
\end{array}\right)\left( \begin{array}{cc}
\alpha_\ast & 0 \\ 0 & 1 
\end{array} \right)^{-1}
\right)\\
& = & \exp\left( \begin{array}{cc}
\frac12 \alpha_\ast \ad_B \alpha_\ast^{-1} & \alpha_\ast (B)\\
0 & 0 \end{array}\right)\\
&=& \varphi_2(\alpha(b))
\end{eqnarray*}
where we used that $\alpha_\ast \ad_B \alpha_\ast^{-1}=\ad_{\alpha_\ast(B)}$ and 
$\alpha_\ast(B)= \log(\alpha(b))$.

\medskip

In order to see that $\varphi$ is injective, it is enough to check that the intersection of the image of $\varphi_1$ and 
$\varphi_2$ contains only the identity element. To see this, it is enough to check that the first $n$ entries of the last column of 
$\varphi_2(g)$ are all equal to 0 if and only if $g=1$, while for any $\alpha \in \Aut(G)$, we have that the first $n$ entries of the last column of $\varphi_1(\alpha)$ are always equal to 0.
\end{proof}

\begin{Rmk}
This theorem is a special case of the more general result that for any nilpotent Lie group $G$, it holds that 
$\Aff(G)$ is a linear group. Also, in the discrete case and in the more general case of polycyclic groups $\Gamma$, it can be shown that $\Gamma\rtimes \Aut(\Gamma)$ is linear (see \cite{ab67-1,baum69-1,merz70-1,wehr74-1} for results in this direction). The advantage of the above theorem~\ref{explicit} is that it provides a very explicit representation $\varphi$ of $\Aff(G)$ as a subgoup of $\Aff(\R^n)$, where $n=\dim(G)$. Moreover, it is easy to that $\varphi(G)$ acts simply transitively on $G$. It follows that if $\Gamma \subseteq \Aff(G)$ is a subgroup acting properly discontinuously (resp. cocompactly) on $G$ then $\varphi(\Gamma)$ acts   properly discontinuously (resp. cocompactly) on $\R^n$.
\end{Rmk}

At this point we want to remark that the 4-dimensional representation that 
was introduced on page \pageref{voordeel} is in fact the  representation $\varphi$ of theorem~\ref{explicit}.
We then further used this representation on page~\pageref{gebruikrep} to show that a certain 
almost-crystallographic group is torsion-free. This was in fact the technique which was used to classify all
almost-Bieberbach groups in dimensions $\leq 4$ in \cite{deki96-1}.

\medskip

Of course having a matrix representation of $\Aff(G)$ for a 2-step nilpotent group is extremely useful to 
study maps on infra-nilmanifolds $\Gamma\backslash G$, since by the generalized second Bieberbach Theorem, 
many maps are constructed by taking an element $(d,\delta)\in \Aff(G)$ such that 
$(d,\delta)\Gamma(d,\delta)^{-1}\subseteq \Gamma$. The obtained matrix representation provides a very effective 
way to make these computations involved in this construction possible.

\section{Expanding maps and Anosov diffeomorphisms on infra--nilmanifolds}

As a first illustration of the use of infra-nilmanifolds, we will see in this section that they  play a crucial role in the study of expanding maps and Anosov diffeomorphisms.

\medskip

We first discuss the case of expanding maps.
\begin{Def}
Let $M$ be a closed smooth Riemannian  manifold. A $C^1$-map $f:M\to M$ is said to be an expanding map\index{expanding map}
if there exist real constants $C>0$ and $\lambda>1$ such that 
\[ \forall v\in TM:\; \forall n\in \N:\; \| Df^n(v)\|\geq C\lambda^n \|v \|.\]
\end{Def}
It can be shown that whether or not a map $f$ is expanding does not depend on the choice of the  Riemannian structure on
$M$.

\medskip

There is a standard way of constructing expanding maps on infra-nilmanifolds.
\begin{Def} Let $M=\Gamma\backslash G$ be an infra-nilmanifold.
An affine map $\overline{(d,\delta)}$ of $\Gamma\backslash G$ is said to be an expanding infra-nilmanifold endomorphism 
of $\Gamma\backslash G$ if and only if for every eigenvalue $\lambda$ of $\delta_\ast$ it holds that 
$|\lambda|>1$.
\end{Def}
The adjective ``expanding'' for these kind of affine maps is well chosen (\cite{shub69-1}):
\begin{Thm}
An expanding infra-nilmanifold endomorphism\index{expanding infra-nilmanifold endomorphism} of an infra-nilmanifold $M$ is indeed an expanding map of that 
infra-nilmanifold.
\end{Thm}

Moreover, as a corollary to his famous theorem on groups of polynomial growth, M.~Gromov showed that 
these kind of expanding maps are essentially (i.e.~up to topological conjugacy) 
the only ones (\cite{grom81-1}, but see also \cite{deki11-1}):
\begin{Thm}
Let $f:M\to M$ be an expanding map on a closed smooth Riemannian manifold $M$, then $f$ is topologically conjugate to an expanding infra-nilmanifold endomorphism. I.e.\ there exists an infra-nilmanifold $\Gamma\backslash G$, an expanding 
infra-nilmanifold endomorphism $\overline{(d,\delta)}$ on $\Gamma\backslash G$ and a homeomorphism
$h:M\to \Gamma\backslash G$, such that the following diagram commutes:
\[\xymatrix{ M \ar[d]_h \ar[r]^f & M \ar[d]^h\\
\Gamma\backslash G\ar[r]_{\overline{(d,\delta)}}& \Gamma\backslash G
}\]
\end{Thm}

This theorem shows that in order to understand the homeomorphism type of those manifolds which admit an expanding map it is enough to study the class of infra-nilmanifolds. We do however remark that F.T.~Farrell and L.E.~Jones constructed examples of expanding maps on exotic tori \cite{fj78-2}. So these manifolds are not diffeomorphic to a torus (or any other infra-nilmanifold), but they are homeomorphic to a torus.

In \cite{es68-1} it was shown that all flat manifolds admit an expanding map and in \cite{ll02-1} this positive result
was extended to all infra-manifolds modeled on a 2-step nilpotent Lie group $G$. This result can not be 
extended to higher nilpotency classes, since it already no longer holds for all nilmanifolds modeled on a 3-step 
nilpotent Lie group. Indeed, in \cite{dl57-1} the authors constructed a 3-step nilpotent Lie algebra $\lie$ all of whose
derivations are nilpotent. It follows that any automorphism $\delta_\ast$ of $\lie$ only has eigenvalues which are 
roots of unity. It follows that none of the nilmanifolds modeled on the corresponding Lie group $G$ (and this 
Lie group does have lattices!), admits an expanding map.

\medskip

It follows that the question of which infra-nilmanifolds $\Gamma\backslash G$ admit an expanding map is a non-trivial one.
In fact only very recently there has been some real progress in this question. It turns out that the question 
whether or not an infra-nilmanifold $\Gamma\backslash G$ admits an expanding map only depends on the 
Lie group $G$ (or the Lie algebra $\lie$) and not on the almost-Bieberbach group $\Gamma$. So either 
all infra-nilmanifolds modeled on $G$ admit an expanding map or none of them.
This was proved in \cite{dere14-1} and \cite{corn14-1} and uses the notion of gradings on a Lie algebra.
A Lie algebra is said to be graded over the integers if $\lie$ admits a vector space 
decomposition as a direct sum $\displaystyle \lie=\bigoplus_{i\in\Z} \lie_i$ such that 
$[\lie_i,\lie_j]\subseteq \lie_{i+j}$ for all $i,j\in \Z$. Such a grading is said to be positive if $\lie_i=0$ for all
$i\leq 0$. Using this terminology, we have:

\begin{Thm}\label{exp-main}
Let $M=\Gamma\backslash G$ be an infra-nilmanifold modeled on the simply connected Lie group $G$ and let $\lie$ be the 
Lie algebra of $G$, then 
\begin{center}
$M$ admits an expanding map\\
$\Updownarrow$ \\
$\lie$ admits a positive grading.
\end{center}
\end{Thm}

Although this result reduces the existence problem of an expanding map on a 
infra-nilmanifold modeled on a given Lie group $G$ completely
 to a pure algebraic condition on the Lie algebra $\lie$ of $G$, it is 
sometimes still difficult to see exactly which Lie algebras do admit such a positive grading and still a lot of work can be 
done on this topic.

\medskip

A somewhat related kind of maps which have been studied are the Anosov diffeomorphisms.
\begin{Def}
Let $M$ be a closed smooth Riemannian  manifold. A $C^1$-diffeomorphism $f:M\to M$ is said to be an Anosov diffeomorphism\index{Anosov diffeomorphism} if and only if there exists a $Df$--invariant continuous splitting of the tangent bundle $TM=E^u \oplus E^s$
and real  constants $C>0$ and $1>\lambda>0$ such that 
\[ \forall v\in E^u :\; \forall n\in \N:\; \| Df^n(v)\|\geq C\lambda^{-n} \|v \|  \mbox{ and }
\forall v\in E^s :\; \forall n\in \N:\; \| Df^n(v)\|\leq \frac1C\lambda^{n} \|v \|.\]
\end{Def}
$E^u$ is called the unstable (or expanding) part and $E^s$ the stable (or contracting) part.
Like in the case of expanding maps, the condition of being an Anosov diffeomorphism does not depend on the chosen Riemannian metric. 

Also for Anosov diffeomorphisms there is an algebraic way of constructing them:

\begin{Def}
Let $M=\Gamma\backslash G$ be an infra-nilmanifold.
An affine diffeomorphism $\overline{(d,\delta)}$ of $\Gamma\backslash G$ is said to be a hyperbolic
 infra-nilmanifold automorphism\index{hyperbolic infra-nilmanifold automorphism} of $\Gamma\backslash G$ if and only if for every eigenvalue $\lambda$ of $\delta_\ast$ it holds that 
$|\lambda|\neq1$.
\end{Def}
\begin{Rmk}
The fact that we require $\overline{(d,\delta)}$ to be a diffeomorphism is equivalent to 
requiring that $\delta$ is invertible and $(d,\delta)\Gamma(d,\delta)^{-1} = \Gamma$. In case $\overline{(d,\delta)}$  is  a hyperbolic
 infra-nilmanifold automorphism  of $\Gamma\backslash G$, then for some of the eigenvalues $\lambda$ of $\delta$ we will have that $|\lambda|>1$, while
 for others we will have that $|\lambda|<1$. 
\end{Rmk}

The following result shows that these hyperbolic infra-nilmanifold automorphisms give us really a  way of constructing 
Anosov diffeomorphisms (\cite{fran70-1} and \cite{smal67-1}):
\begin{Thm}
A hyperbolic infra-nilmanifold automorphism of an infra-nilmanifold $M$ is an Anosov diffeomorphism of that 
infra-nilmanifold.
\end{Thm}

The most famous example of such a hyperbolic infra-nilmanifold endomorphism is most probably Arnold's cat map,
which is the map induced on the 2-dimensional torus $T^2=\Z^2\backslash \R^2$ by the linear 
map $\delta:\R^2 \to \R^2$, where $\delta$ is given by the matrix 
$\left( \begin{array}{cc} 2 & 1 \\ 1 & 1 \end{array}\right)$.

There is a long standing conjecture (\cite{fran70-1}) that in fact the hyperbolic infra-nilmanifold  endomorphisms are essentially the only examples of Anosov diffeomorphisms:

\begin{con}
Let $f:M\to M$ be an Anosov diffeomorphism on a closed manifold $M$, then $f$ is topologically conjugate to a
hyperbolic  infra-nilmanifold automorphism.
\end{con}

Having this conjecture in mind it is natural to ask which infra-nilmanifolds $M$ admit an Anosov diffeomorphism.
It has been shown that, for infra-nilmanifolds, admitting an Anosov diffeomorphism is indeed equivalent to admitting a
hyperbolic infra-nilmanifold automorphism (\cite{deki11-1}).

The following lemma is quite easy to prove:
\begin{Lem}\label{anotorus}
A $n$-dimensional torus admits an Anosov diffeomorphism if and only if $n\geq 2$.
\end{Lem} 
\begin{proof}
That $n$ must be at least 2 is obvious, since the tangent bundle has to split in an expanding part $E^u$ and a
contracting part $E^s$. To show that each $n$-dimensional torus, with $n\geq 2$, admits an Anosov diffeomorphism,
 it is enough to show that 
there exists a matrix $\delta_n\in \GL(\Z^n)$ with no eigenvalue of modulus 1. For $n=2$, we can take $\delta_2$ to be the 
matrix of Arnold's cat map given above. For $n=3$ we can e.g.\ take 
\[ \delta_3 = \left( \begin{array}{ccc} 0 & 0 & 1\\ 1 & 0 & 0 \\ 0 & 1 & 1\end{array}\right). \]
For higher values of $n$, we can choose a block diagonal matrix with blocks $\delta_2$ and/or $\delta_3$ on the diagonal.
\end{proof}

There has been a lot of research on the existence question of Anosov diffeomorphisms for nilmanifolds, especially in low dimensional cases or for 
nilmanifolds modeled on very special types of nilpotent Lie groups (See e.g. \cite{cks99-1,dani78-1,  deki99-1,  dm05-1,  dd03-2, dv09-1,gorb05-1,laur03-1,laur08-1,  lw08-1, lw09-1, main06-1, main12-1,
malf97-3, mw07-1, payn09-1}). The situation is much more difficult than in the case of expanding maps. As an 
illustration of how complicated the situation can be let us remark here that for $G=H^\R\times H^\R$, the direct product 
of the Heisenberg Lie group with itself, there are lattices $N_1,N_2 \subseteq G$ such that 
the nilmaniolfd $N_1\backslash G$ admits an Ansosov diffeomorphism, while 
$N_2\backslash G$ does not admit an Anosov difeomorphism (see  \cite{cks99-1,malf97-3}). 
So, in the case of Anosov diffeomorphisms there is no hope to find a result in the sense of Theorem~\ref{exp-main}.

\medskip

As it is already quite complicated to study Anosov diffeomorphisms on nilmanifolds, not that much research has been done
in the case of infra-nilmanifolds. Nevertheless there is a complete answer in the flat case, due to H.~Porteous (\cite{port72-1}).
\begin{Thm}
Let $\Gamma \subseteq \Isom(\R^n)$ be a Bieberbach group with rational holonomy representation 
$\varphi: F \to \GL(\Q^n)$ (see page~\pageref{ratholon}) and let $M=\Gamma\backslash \R^n$ be the 
corresponding flat manifold. Then we have:
\begin{center}
$M$ admits an Anosov diffeomorphism \\
$\Updownarrow$\\
Each $\Q$--irreducible component of $\varphi$ which appears with multiplicity one is reducible over $\R$.
\end{center}
\end{Thm}  
When we want to generalize the above theorem to the case of infra-nilmanifolds, we need to restrict ourselves 
to certain classes of them e.g.\ by considering only infra-nilmanifolds modeled on well understood 
nilpotent Lie groups $G$. A successful generalization was obtained in case $G$ is a free nilpotent group.
\begin{Def}
A Lie group $G$ is said to be a free nilpotent Lie group of class $c$ on $r$ generators in case 
its Lie algebra $\lie$ is a free nilpotent Lie algebra of class $c$ on $r$ generators. We denote the 
free nilpotent Lie group of class $c$ on $r$ generators by $G_{c,r}$.
\end{Def} 
We have a complete understanding of which nilmanifolds modeled on a free nilpotent Lie group $G_{c,r}$ do 
admit an Anosov diffeomorphism. It turns out that this only depends on the values of $c$ and $r$ and not on the 
specific lattice which was chosen to construct the nilmanifold (\cite{dani78-1, deki99-1, dv09-1})
\begin{Thm}
Let $N$ be a lattice in the free $c$-step nilpotent Lie group $G_{c,r}$ on $r$ generators. Then, the 
nilmanifold $N\backslash G_{c,r}$ admits an Anosov diffeomorphism if and only if $r>c$.
\end{Thm}
When $c=1$, then $G_{1,r}=\R^r$ and the above theorem says that a torus $N\backslash \R^r$ (with 
$N\cong \Z^r$ a lattice in $\R^r$) admits an Anosov diffeomorphism if and only if $r>1$. This is exactly the content of 
Lemma~\ref{anotorus}.

The Heisenberg group $H^\R$ is  equal to $G_{2,2}$. The theorem above says that no 
nilmanifold (and hence also no infra-nilmanifold) 
modeled on the Heisenberg group admits an Anosov diffemorphism. Actually, no infra-nilmanifold, which is not a flat 
manifold, in dimension $\leq 5$ admits an Anosov diffeomorphism (see e.g.\ \cite{malf97-3}).

Based on  \cite{dv08-1} and \cite{dv11-1}, we were very recently able to prove in \cite{dd13-1} a full generalization of 
Porteous' result to the case of infra-nilmanifolds modeled on a free nilpotent Lie group.

\begin{Thm}
Let $\Gamma\subseteq G_{c,r}\semi F$ be an almost--Bieberbach group modeled on the free $c$-step nilpotent Lie group 
with $r$ generators and let $\varphi: F\to \GL(\Q^r)$ be the associated abelianized rational holonomy representation (see page~\pageref{abrathol}), then 
\begin{center}
The infra-nilmanifold $\Gamma\backslash G_{c,r}$ admits an Anosov diffeomorphism\\
$\Updownarrow$\\
Every $\Q$--irreducible component of $\varphi$ that occurs with multiplicity $m$, splits in strictly more than $\frac{c}{m}$ components when viewed as a representation over $\R$. 
\end{center}
\end{Thm}
Let us see what this condition says in case $c=1$. In that case, we are dealing with flat manifolds and the 
abelianized rational holonomy reresentation is just the rational holonomy representation. \\
The above theorem says that a $\Q$-irreducible component that appears with multiplicity $m=1$ has to split in more than 
$1$ component over $\R$ (so is reducible over $\R$), while a component that appears with multiplicity $m>1$ has to 
split in more than $1/m<1$ components, which means it does not need to split (so there are no conditions on these components). This is exactly the same result as what H.~Porteous proved.

\medskip

Apart from the above, no real general results on the existence of Anosov diffeomorphism on infra-nilmanifolds has been 
obtained and still  a lot of research can be done in this direction.

\medskip

In general one has the feeling that admitting an Anosov diffeomorphism is a much stronger condition than 
admitting an expanding map. However, recently J.~Der\'e constructed an example of a nilmanifold admitting 
an Anosov diffeomorphims, but no expanding map (\cite{dere14-1}).

\section{Nielsen fixed point theory on infra--nilmanifolds}
A second domain in which infra-nilmanifolds have been studied quite a lot  is that of topological fixed point theory, in particular Nielsen fixed point theory. 

In this section, we consider 
a compact manifold $X$ (more general spaces are allowed in Nielsen theory, but in our case it suffices to consider only 
compact manifolds).

Let $f:X\to X$ be a selfmap and let $\Fix(f)=\{x\in X \; |\; f(x)=x\}$ denote the fixed point set of $f$. The goal of Nielsen 
fixed point theory is to study the minimal number of fixed points of all maps $f'$ which are homotopic to $f$. 

A first indication on the number of fixed points is the Lefschetz number\index{Lefschetz number} of $f$, denoted by $L(f)$, which is defined as
the alternating sum of the traces of the linear maps induced by $f$ on the homology spaces $H_i(X,\Q)$ (which are finite
dimensional rational vectorspaces):
\[ L(f) = \sum_{i=0}^{\dim X} (-1)^i {\rm Trace} (f_{\ast,i}: H_i(X,\Q) \to H_i(X,\Q)\,).\]
The famous Lefschetz fixed point theorem states that in case $L(f)\neq 0$, then $f$ has at least one fixed point.
Moreover, since $L(f')=L(f)$ for all maps $f'$ which are homotopic to $f$, the non-vanishing of $L(f)$ ensures that 
any map $f'$ homotopic to $f$ has at least one fixed point. Unfortunately, the Lefschetz number does not give any 
information on the exact number of fixed points. In fact, it is even possible that $L(f)=0$, while any map homotopic to $f$ 
does have at least one fixed point.

There is a second number, called the Nielsen number\index{Nielsen number} of $f$, denoted by $N(f)$, also homotopy invariant and which contains more information than $L(f)$. Unfortunately, $N(f)$ is in general more difficult to compute than $L(f)$. 

To define the Nielsen number of $f$, we divide $\Fix(f)$ into fixed point classes\index{fixed point class}. We say that two elements 
$x,y\in \Fix(f)$ are Nielsen equivalent if there exists a path $\alpha:[0,1]\to X$ with $\alpha(0)=x$, $\alpha(1)=y$ and such 
that $\alpha$ and $f\circ \alpha$ are path homotopic: 
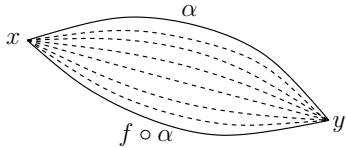
\begin{center}
\includegraphics[width=0.3\textwidth]{nielsen.0}
\end{center}
It is easy to see that being Nielsen equivalent is an equivalence relation on the set $\Fix(f)$. The equivalence classes are called the fixed point classes of $f$. To each fixed point class ${\mathbf F}$, one associates an integer $I(f,{\mathbf F})$, the fixed point index of ${\mathbf F}$. This is done in an axiomatic way (see \cite{brow71-1,jian64-1}) and is not so easily explained in a few words. To give some idea about the index, we can mention that when $X$ is a differentiable manifold, $f:X\to X$ is a differentiable function and ${\mathbf F}=\{x_0\}$ is a fixed point class consisting of one isolated fixed point, then 
\[ I(f,{\mathbf F})={ \rm sgn} \det(1 - df_{x_0})\]
where $df_{x_0}:T_{x_0}X\to T_{x_0}X$ is the differential of $f$ at $x_0$, $1$ denotes the identity map of $T_{x_0}X$ and 
sgn$(r)=-1, 0$ or $1$ when $r<0$, $r=0$ and $r>0$ respectively.

A fixed point class is said to be 
essential  when the index $I(f,{\mathbf F})\neq 0$. The idea is that essential fixed point classes cannot vanish under a homotopy, while unessential ones might disappear (become empty). The Nielsen number of $f$ is defined as
\[ N(f)=\mbox{ the number of essential fixed point classes of $f$} =\# \{ {\mathbf F}\;|\; 
 I(f,{\mathbf F})\neq 0 \}.\]
It is obvious that $N(f)\leq \# \Fix(f)$. Moreover, it can be shown that $N(f)$ is a homotopy invariant, from which 
it follows that $N(f) \leq \#\Fix(f')$ for all maps $f'$ which are homotopic to $f$, so we have
\[  N(f) \leq {\rm min} \{ \# \Fix(f')\;|\; f'\sim f \} .\]
So $N(f)$ is a lower bound for the number of fixed points of any map homotopic to $f$. It turns out that in many cases
this lower bound is sharp, since there is the following theorem due to Wecken (\cite{weck42-1}):
\begin{Thm}
Let $f:X\to X$ be a map on a closed manifold of dimension $\geq 3$. Then there exists a map $f'$ homotopic to $f$  such that 
$N(f)=\# \Fix(f')$.
\end{Thm}
But although the circle, the torus and the Klein bottle, the only infra-nilmanifolds in dimensions $<3$, are not included in the above theorem of Wecken, one can still show (since it is easy to determine all maps of these manifolds up to homotopy), that the result of Wecken's theorem is still valid for these manifolds. 
So in case we are working with infra-nilmanifolds, the Nielsen number $N(f)$ is always a sharp lower bound for the number of fixed points of maps homotopic to $f$.

For maps on nilmanifolds, D. Anosov proved that the computation of $N(f)$ is as easy as the computation of the 
Lefschetz number (\cite{anos85-1}):
\begin{Thm}
Let $G$ be a simply connected nilpotent Lie group and let $N$ be a lattice of $G$. \\
Assume that $f:N\backslash G \to N\backslash G$ is a map inducing 
a homomorphism $f_\ast :N \to N$. Let $\delta:G \to G$ be the unique extension of $f_\ast$ to 
a homomorphism $\delta:G\to G$ and let $\delta_\ast:\lie\to \lie$ be the differential of $\delta$, then 
\[ L(f)= \det (1-\delta_\ast) \mbox{ and } N(f)=| \det (1-\delta_\ast)|.\]
\end{Thm}

\begin{Rmk}
Note that $\delta$ is a homomorphism such that $f$ is homotopic to $\overline{(1,\delta)}$ (see section~\ref{maps}).\\
The reader who is familiar with Nomizu's work (\cite{nomi54-1}) can see where the formula for the Leftschetz number comes  from. 
\end{Rmk}

So for nilmanifolds, we always have that $N(f)=|L(f)|$ and hence in this case the Lefschetz number does contain 
all the information we want. 
\begin{Def}
Let $f:X\to X$ be a selfmap of a closed manifold. We say that $f$ satisfies the Anosov relation when $N(f)=|L(f)|$. \\
We say that a manifold $X$ satisfies the Anasov relation, in case the Anosov relations holds for any selfmap $f$ of $X$
\end{Def}

We have just seen that any nilmanifold satisfies the Anosov relation. This is no longer true for infra-nilmanifolds as we shall see below.

For maps on infra-nilmanifolds, it suffices to consider affine maps, since any map is homotopic to an affine 
map. By using the fact that any infra-nilmanifold is finitely covered by a nilmanifold S.W.~Kim, J.B.~Lee and K.B.~Lee 
were able to prove a nice averaging formula to compute both the Lefschetz and the
 Nielsen number of a selfmap of an infra-nilmanifold 
(see \cite{kll05-1} and \cite{ll06-1}).
\begin{Thm}\label{averaging} Let $G$ be a simply connected nilpotent Lie group. Assume that 
$\Gamma \subseteq G\semi C$ (with $C$ a compact subgroup of $\Aut(G)$) is an almost--Bieberbach group 
with holonomy group $F$ (where we view $F\subseteq C\subseteq \Aut(G)$ as the subgroup of rotational parts of the 
elements of $\Gamma$). \\  
Assume that $f=\overline{(d,\delta)}$ is an affine selfmap of the infra-nilmanifold $\Gamma\backslash G$, then 
\[ L(f)=\frac{1}{|F|} \sum_{\alpha\in F} \det(1-\alpha_\ast \delta_\ast)\mbox{ \ \ and \  \ } 
 N(f)=\frac{1}{|F|} \sum_{\alpha\in F} |\det(1-\alpha_\ast \delta_\ast)|\]
\end{Thm}
From this formula it is clear  that the Anosov relation holds if and only if  the terms $\det(1-\alpha_\ast \delta_\ast)$ are 
either all non-negative or all non-positive. We formulate this as a corollary:
\begin{Cor}
Let $f=\overline{(d,\delta)}$ be as in Theorem~\ref{averaging}, then 
\[N(f)=|L(f)|\Leftrightarrow \forall \alpha,\beta \in F: \det(1-\alpha_\ast \delta_\ast)\det(1-\beta_\ast \delta_\ast)\geq 0\]
\end{Cor}

Let us illustrate the averaging formula on the Klein bottle. Recall that the Klein bottle was constructed 
as a quotient $\Gamma\backslash \R^2$ where 
$\Gamma$ is generated by 
\[ a=\left( \begin{array}{ccc} 1 & 0 & 1\\ 0 & 1& 0 \\ 0 & 0 & 1
\end{array}\right), \; 
b=\left( \begin{array}{ccc} 1 & 0 & 0\\ 0 & 1& 1 \\ 0 & 0 & 1
\end{array}\right) \mbox{ and }
\alpha=\left( \begin{array}{ccc} -1 & 0 & 0\\ 0 & 1& \frac12 \\ 0 & 0 & 1
\end{array}\right)\]
where we represent elements $(a,A)$ of $\Aff(\R^2)$ as a $3\times 3$ matrix $
\left(\begin{array}{cc} A & a \\ 0 & 1 \end{array} \right) $ as in section~\ref{computation}.

Consider the affine map $(d,\delta)=\left( \begin{array}{ccc}
2 & 0 & -\frac12 \\
0 & 3 & 0 \\
0 & 0 & 1\end{array} \right)$ (also seen as a $3\times 3$-matrix). 
We easily compute that 
\[ (d,\delta) a (d,\delta)^{-1} = a^2,\; (d,\delta) b (d,\delta)^{-1} = b^3\mbox{ and }
(d,\delta) \alpha (d,\delta)^{-1} = a^{-1} b \alpha .\]
This implies that $(d,\delta) \Gamma (d,\delta)^{-1} \subseteq \Gamma$ and so
$(d,\delta)$ induces a map $f=\overline{(d,\delta)}$ on the Klein bottle $\Gamma\backslash \R^2$.
To compute $L(f)$ and $N(f)$ note that $\delta_\ast=\delta$ and that
 \[F=\left\{\alpha_1=\alpha_{1\ast}=I_2,\;
\alpha_2=\alpha_{2\ast}=\left( \begin{array}{cc} -1 & 0 \\0 & 1\end{array} \right)
 \right\}\]
  and use this in the following computation involving the averaging formulas of Theorem~\ref{averaging}:
\begin{eqnarray*}
L(f) & = & \frac12\left[ \det\left( I_2-  \left( \begin{array}{cc} 2 & 0 \\ 0 & 3 \end{array}\right)\right) +
                         \det\left( I_2 - \alpha_2 \left( \begin{array}{cc} 2 & 0 \\ 0 & 3 \end{array}\right)\right)    \right]\\
& = &  \frac12\left[ \det\left( I_2 -  \left( \begin{array}{cc} 2 & 0 \\ 0 & 3 \end{array}\right)\right) +
                         \det\left( I_2 -  \left( \begin{array}{cc} -2 & 0 \\ 0 & 3 \end{array}\right)\right)    \right]\\
& = & \frac12 \left[ \det\left( \begin{array}{cc} -1 & 0 \\ 0 & -2 \end{array}\right) + \det\left( \begin{array}{cc} 3 & 0 \\ 0 & -2 \end{array}\right) \right]\\
& = & \frac12[ 2 +(-6)]= -2\\[2mm]
 N(f) & = & \frac12[ |2| +|-6|]=4.
\end{eqnarray*}
One clearly sees that this $f$ does not satisfy the Anosov relation.

\medskip

Let us conclude this section by discussing some situations for which the 
Anosov relation does hold. We can do this by looking for conditions on the map or on the manifold.

\medskip
Let us first use the averaging formula to give a new and short proof for the Anosov relation in case 
of  homotopically periodic maps on infra-nilmanifolds. A map $f:X\to X$ is said to be homotopically periodic if
 there exists a positive integer $k$ such that $f^k$ is homotopic to the identity.
The following theorem was first proved by S.~Kwasik and K.B.~Lee in 1988 (\cite{kl88-1}).

\begin{Thm}
Let $f:X\to X$ be a homotopically periodic map on an infra-nilmanifold $X$, then $N(f)=L(f)$.
\end{Thm}
\begin{proof}
Let $X=\Gamma\backslash G$ where $\Gamma$ is an almost-Bieberbach group modeled on a nilpotent Lie group $G$.
We can assume that $f$ is homotopic to an affine map $\overline{(d,\delta)}$, where $(d,\delta)\in \Aff(G)$.
The map $(d,\delta)$ satisfies
\[ \forall \gamma  \in \Gamma: \; f_\ast (\gamma) (d,\delta) = (d,\delta) \gamma.\]
Let $k$ be the positive integer for which $f^k$ is homotopic to the identity, then $f_\ast^k=(d,\delta)^k$ is 
the identity, from which it follows that $\delta^k=1_G$. It follows that $\delta$ is invertible, and so 
$\forall \gamma=(a,\alpha) \in \Gamma:\; f_\ast (\gamma) = (d,\delta) \gamma (d,\delta)^{-1}$. From this, it  follows that 
$F=\delta F \delta^{-1}$. Now, let $\alpha\in F$, then $\exists \alpha'\in F$:
\[ (\alpha \delta)^{k|F|} =(\alpha'\delta^k)^{|F|}=\alpha'^{|F|}=1.\]
So, any $\alpha_\ast \delta_\ast$ is a matrix of finite order, having only roots of unity $\lambda$ as eigenvalues.
Hence $\det(I_n-\alpha_\ast \delta_\ast)=\prod_{i=1}^n(1 -\lambda_i)$ with $|\lambda_i|=1$, We claim that 
this determinant is always $\geq 0$. Indeed, we have the following possibilities for $\lambda_i$:
\begin{itemize}
\item $\lambda_i=1$, this gives a contribution $1-1=0$ to the product (so when $\alpha_\ast \delta_\ast$ has 1 as an eigenvalue, the determinant $\det(I_n-\alpha_\ast \delta_\ast)=0$).
\item $\lambda_i=-1$, this gives a contribution $1-(-1)=2>0$ to the product.
\item $\lambda_i\in \C\backslash \R$, then also its complex conjugate $\overline{\lambda_i}$ is an eigenvalue
of $\alpha_\ast \delta_\ast$. The contribution of the pair $\{\lambda_i, \overline{\lambda_i} \}$ to the 
product is $(1-\lambda_i)(1-\overline{\lambda_i})=2 -2 {\rm Re}(\lambda_i)>0$.
\end{itemize}
It follows that for $\alpha\in F$, $\det(I_n-\alpha_\ast\delta_\ast) \geq 0$, from which we can deduce that $L(f)=N(f)$.
\end{proof}

Analogously, one can prove the following result on expanding maps which was first obtained in \cite{ddm05-1}.
\begin{Thm}
Let $f$ be an expanding map on an infra-nilmanifold $X$, then 
\[ N(f)=L(f) \Leftrightarrow X\mbox{ is orientable.}\]
In case $X$ is not orientable then $N(f)\neq |L(f)|$.
\end{Thm}
\begin{Rmk}
An infra-nilmanifold $X=\Gamma\backslash G$ is orientable if and only if $\forall \alpha \in F:\;\det(\alpha_\ast)=1$.
\end{Rmk}

The previous two results were about the Anosov relation for specific types of maps. Now we list some results concerning specific manifolds. The first theorem was proved in \cite{ddm05-1}: 
\begin{Thm}
Let $\Gamma\backslash G$ be an infra-nilmanifold with an odd order holonomy group $F$, then 
$N(f)=|L(f)|$ for any selfmap $f$ of $\Gamma\backslash G$.
\end{Thm}
Note that any infra-nilmanifold with an odd order holonomy group is orientable, so this is compatible with the 
previous result.

\medskip

As a final illustration, we want to mention a result on infra-nilmanifolds with a holonomy group 
of even order. In any dimension $n\geq 2$ there exist flat manifolds with a holonomy group 
isomorphic to  $\Z_2^{n-1}$ and there are no flat manifolds with a holonomy group isomorphic 
to $\Z_2^k$ with $k\geq n$. The number of such manifolds grows exponentially with the dimension $n$ 
(\cite{mr99-1}). In dimension 2, there is the non-orientable Klein bottle, in dimension 3, there is 
one orientable flat manifold with holonomy group $\Z_2^2$ and two non-orientable ones. The orientable one is 
known as the Hantsche-Wendt manifold. Therefore, one refers to a $n$-dimensional flat manifold with holonomy group $\Z_2^{n-1}$ as a generalized Hantsche-Wendt manifold. For these manifolds, we have proven in \cite{ddm04-1}
the following theorem:
\begin{Thm}
Let $X$ be a generalized Hantsche-Wendt manifold, then 
\begin{itemize}
\item $N(f)=|L(f)|$ for all selfmaps $f$ of $X$ in case $X$ is orientable.
\item In case $X$ is non-orientable, there exists a selfmap of $X$ such that $N(f)\neq |L(f)|$ (e.g.\ one can take $f$ to be an expanding map).
\end{itemize}
\end{Thm} 

Of course, many more results about Nielsen fixed point theory for infra-nilmanifolds have been proved and the above 
theorems are mainly meant to illustrate what kind of results can be obtained.

 Besides fixed point theory, one has also studied (and is still studying) periodic points and coincidences.
 \begin{Def} 
 Let $f:X\to X$ and $g:X\to Y$ be a maps. \\
 A periodic point of $f$ is a point $x\in X$ such that $f^n(x)=x$ for some positive integer $n$.\\
 A pair of points $(x,y)$ with $x,y\in X$ is called a coincidence pair of $g$ if $g(x)=g(y)$.
 \end{Def}
 In literature one can find quite some results on Nielsen periodic point theory
  (\cite{heat99-1,hk97-1,hk00-1,hk02-1,jkm04-1,jezi03-1,jm02-1,jm06-1,hkl10-1,kl88-1}) and Nielsen coincidence theory (\cite{dp10-1, gonc98-1,gw01-1, gw05-1,hlp11-1, hlp12-1, jezi90-1,kl05-1,kl07-1, wong04-1} for maps on/between infra-nilmanifolds. 

%\bibliography{/home/u0006953/algebra/ref}
%\bibliographystyle{/home/u0006953/algebra/ref}

%\printindex

\end{document}